\documentclass[10pt,reqno]{amsart}

\usepackage{amsmath}
\usepackage{mathtools}
\usepackage{amssymb}
\usepackage{graphicx}
\usepackage{color}
\usepackage{latexsym}
\usepackage[T1]{fontenc}

\newcommand{\NN}{{\mathbb N}} % Natural numbers
\newcommand{\QQ}{{\mathbb Q}}  % Rationals
\newcommand{\sym}{\mathcal{S}}

\newcommand{\Invs}[1]{\mathcal{I}_{#1}}

\newcommand{\pattern}{
  \begin{minipage}[c]{1.45em}\scalebox{0.5}{\includegraphics{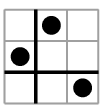}}
  \end{minipage}
}

\newcommand{\lara}{\ensuremath{3{\bar 1}52\bar{4}}}

\newcommand{\si}{\sigma}

\def\emm#1,{{\em #1}}

\DeclareMathOperator{\Asc}{\mathfrak{asc}}
\DeclareMathOperator{\st}{std}
\DeclareMathOperator{\comp}{comp}
\DeclareMathOperator{\last}{last}
\DeclareMathOperator{\asc}{asc}
\DeclareMathOperator{\lmax}{lmax}
\DeclareMathOperator{\lmin}{lmin}
\DeclareMathOperator{\rmax}{rmax}
\DeclareMathOperator{\rmin}{rmin}
\DeclareMathOperator{\rlmax}{rl-max}
\DeclareMathOperator{\zeros}{zeros}

\newcommand{\addi}{\textsf{Add1}}
\newcommand{\addii}{\textsf{Add2}}
\newcommand{\addiii}{\textsf{Add3}}

\newcommand{\subi}{\textsf{Rem1}}
\newcommand{\subii}{\textsf{Rem2}}
\newcommand{\subiii}{\textsf{Rem3}}

\newcommand{\rank}{\ell} % rank (but not in the normal poset terminology)
\newcommand{\srank}{\ell^{\star}} % special rank
\newcommand{\p}{permutation}

\newcommand{\add}{\varphi}
\newcommand{\remove}{\psi}

\newcommand{\M}{\mathcal{M}}
\newcommand{\N}{\mathcal{N}}
\newcommand{\A}{\mathcal{A}}      % ascent sequences 
\renewcommand{\P}{\mathcal{P}}    % (2+2)-free posets
\newcommand{\R}{\mathcal{R}}      % Our set of permutations
\newcommand{\nyip}{\hspace*{0.425em}}

\newtheorem{theorem}{Theorem}
\newtheorem{proposition}[theorem]{Proposition}
\newtheorem{corollary}[theorem]{Corollary}

\newtheorem{lemma}[theorem]{Lemma}

\theoremstyle{definition}
\newtheorem{example}{Example}
\newtheorem{question}{Question}

\newcommand{\beq}{\begin{equation}}
\newcommand{\eeq}{\end{equation}}

\newcommand{\gf}{generating function}

\newcommand{\fps}{formal power series}

\newcommand{\tpt}{$(\mathbf{2+2})$}
\newcommand{\tptp}{$\mathbf{2+2}$}

\graphicspath{{Figures/}}
\begin{document}

\title[Posets, sequences and permutations]
{{\tpt}-free posets, ascent sequences \\ 
and pattern avoiding permutations}

\author[M. Bousquet-M\'elou]{Mireille Bousquet-M\'elou}
\address{M. Bousquet-M\'elou: CNRS, LaBRI, Universit\'e Bordeaux 1, 
351 cours de la Lib\'eration, 33405 Talence, France}  

\thanks{MBM was supported by the French ``Agence Nationale
de la Recherche'', project SADA ANR-05-BLAN-0372.}

\author[A. Claesson]{Anders Claesson}
\address{A. Claesson and S. Kitaev: The Mathematics Institute,
Reykjavik University, 103 Reykjavik, Iceland} 
 \thanks{AC and SK were supported by grant no. 060005012 from
  the Icelandic Research Fund.}

\author[M. Dukes]{\\ Mark Dukes}
\address{M. Dukes: Science Institute, University of Iceland, 107
Reykjavik, Iceland}

\author[S. Kitaev]{Sergey Kitaev}

\date{November 25th, 2009}

%%%%%%%%%%%%%%%%%%%%%%%%%%%%%%%%%%%%%%%%%%%%%%%%%%%%%%%%%%%%%%%
\begin{abstract}
We present bijections between four classes of combinatorial
objects. Two of them, the class of unlabeled {\tpt}-free posets and a
certain class of involutions (or chord diagrams), already appeared in
the literature, but were apparently not known to be equinumerous.  We
present a direct bijection between them.  The third class is a family
of permutations defined in terms of a new type of pattern. An
attractive property of these patterns is that, like classical
patterns, they are closed under the action of the symmetry group of
the square. The fourth class is formed by certain integer sequences,
called ascent sequences, which have a simple recursive structure and
are shown to encode \tpt-free posets and permutations.  Our bijections
preserve numerous statistics.

We  determine the \gf\ of these classes of objects, thus
recovering a non-D-finite series obtained by Zagier for the class of
chord diagrams.  
Finally,
we characterize the ascent sequences that correspond to permutations
avoiding the barred pattern $3{\bar 1}52{\bar 4}$ and use this to
enumerate those permutations, thereby settling a conjecture of 
Pudwell.

\end{abstract}

%%%%%%%%%%%%%%%%%%%%%%%%%%%%%%%%%%%%%%%%%%%%%%%%%%%%%%%%%%%%%%%

\maketitle
\thispagestyle{empty}

%%%%%%%%%%%%%%%%%%%%%%%%%%%%%%%%%%%%%%%%%%%%%%%%%%%%%%%%%%%%%%%%
\section{Introduction}
%%%%%%%%%%%%%%%%%%%%%%%%%%%%%%%%%%%%%%%%%%%%%%%%%%%%%%%%%%%%%%%%
This paper presents correspondences between three main
structures, seemingly unrelated:
unlabeled {\tpt}-free posets on $n$ elements, 
certain fixed point free involutions (or chord diagrams) on $2n$
elements introduced by Stoimenow in connection with Vassiliev
invariants of knots~\cite{stoim},
and a new class of permutations on $n$ letters.
An auxiliary class of objects, consisting of certain sequences of
nonnegative integers that we call \emph{ascent sequences}, plays a
central role in some of these
correspondences. Indeed, we show that both our permutations and
{\tpt}-free posets  can be encoded as  ascent sequences.

A poset is said to be \emph{{\tpt}-free} if it does not contain an
induced subposet that is isomorphic to {\tptp}, the union of two
disjoint 2-element chains. Fishburn~\cite{fishburn} showed that a poset
is {\tpt}-free precisely when it is isomorphic to an interval order.
Amongst other results concerning {\tpt}-free posets
\cite{FISH_BOOK,FISH_OPER, SKANDERA, ZAHAR}, the following
characterisation plays an important role in this paper: a poset is
{\tpt}-free if and only if the collection of strict principal
down-sets can be linearly ordered by inclusion \cite{bogart}.
Precise definitions will be given in Sections~\ref{sec:poset} and~\ref{sec:chord}. 

The class of permutations we consider will be defined in
Section~\ref{sec:asc-av}, 
together with ascent sequences.  
Essentially, it is a class of permutations that {\em{avoid}} a
particular pattern of length three.  This type of pattern is new in
the sense that it does not admit an expression in terms of the
vincular\footnote{Babson and Steingr\'imsson call these patterns
``generalized'' rather than ``vincular'', but we wish to promote a
change of terminology here, since vincular is more descriptive. The
adjective vincular is derived from the Latin noun vinculum (``bond''
in English).} patterns introduced by Babson and
Steingr\'imsson~\cite{BABSON_EINAR}.
An attractive property of these new patterns is that, like classical
patterns, they are closed under the action of the symmetry group of
the square. Vincular patterns do not enjoy this property.  We show how
to construct (and deconstruct) these permutations element by element,
and how this gives a bijection $\Lambda$ with ascent sequences.

  \begin{figure}[t!]
   \input{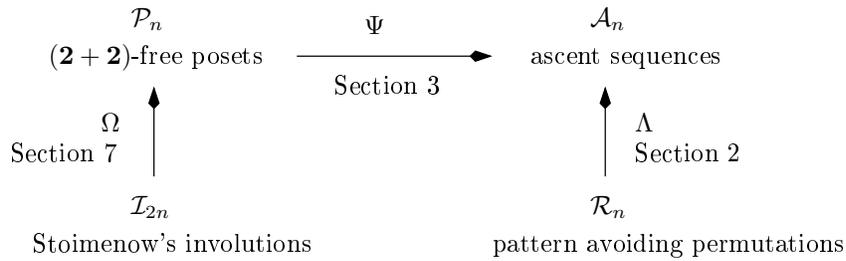}
    \caption{The bijections of the paper.}
    \label{fig:outline}
  \end{figure}

In Section~\ref{sec:poset} we 
perform a similar task for {\tpt}-free posets. We present a recursive
construction of these posets, more sophisticated than that of permutations,
which gives a bijection $\Psi$ with ascent sequences.

In Section~\ref{sec:mod} we present a simple algorithm that, given an
ascent sequence $x$, computes what we call the modified ascent
sequence, denoted $\widehat x$. Some of the properties of the 
permutation and the poset corresponding to $x$ are
more easily read from $\widehat x$ than from $x$. We also explain how
to go directly between a given poset and the corresponding permutation
as opposed to via the ascent sequence. As an additional application of
our machinery we show that the fixed points under $x\mapsto\widehat x$
are in one-to-one correspondence with permutations avoiding the barred
pattern $3{\bar 1}52{\bar 4}$. We use this characterization to
count these permutations,  thus proving a conjecture
of  Pudwell~\cite{lara}.

In Section~\ref{sec:stat} we prove that the bijections
$\Lambda$ and $\Psi$ respect numerous natural statistics.

In Section~\ref{sec:gf} we determine the \gf\ of ascent sequences, and
thus, of \tpt-free posets and pattern avoiding permutations.
Several authors have tried to count these posets
before~\cite{haxell,ZAHAR,smkhamis}, but did not obtain a closed
expression for the \gf, which turns out to be a rather
complicated, non-D-finite series.  
That our approach succeeds probably relies on the simple structure of
ascent sequences.  

The \gf\ we obtain for \tpt-free posets has, however, already appeared in the literature:
it was shown by Zagier~\cite{zagier} to count certain involutions (or
chord diagrams) introduced by Stoimenow to give upper
bounds on the  
dimension of the space of Vassiliev's knot invariants of a given
degree~\cite{stoim}. In Section~\ref{sec:chord} we present an
alternative proof of Zagier's result by giving 
 a direct bijection $\Omega$ between \tpt-free posets and Stoimenow's
 involutions. 

Finally,  in Section~\ref{sec:open} we state some  natural questions.

Let us conclude with a few words on the genesis of this paper: we
started with an investigation of permutations avoiding our new type of
pattern. Patterns of length 2 being trivial, we moved to length 3, and
discovered that the numbers counting one of our permutation classes
formed the rather mysterious sequence A022493 of the on-line
Encyclopedia of Integer Sequences~\cite{sloane}. From this arose the
curiosity to clarify the connections between this class of
permutations and
\tpt-free posets, but also between these posets and Stoimenow's
involutions, as this had apparently not been done before.  We hope
that the study of these new pattern-avoiding permutations will lead to
other connections with interesting objects.

%%%%%%%%%%%%%%%%%%%%%%%%%%%%%%%%%%%%%%%%%%%%%%%%%%%%%%%%%%%%%%%%
\section{Ascent sequences and pattern avoiding permutations}
\label{sec:asc-av}
%%%%%%%%%%%%%%%%%%%%%%%%%%%%%%%%%%%%%%%%%%%%%%%%%%%%%%%%%%%%%%%%

Let $(x_1,\dots , x_i)$ be an integer sequence. The number of
{\em{ascents}} of this sequence is
\begin{equation*}
  \asc(x_1,\dots , x_{i}) = |\{\,1\leq j <i\,:\, x_j<x_{j+1}\,\}|.
\end{equation*}
Let us call a sequence $x=(x_1,\dots , x_n ) \in \mathbb{N}^n$ an
{\em{ascent sequence of length $n$}} if it satisfies $x_1=0$ and $x_i
\in [0,1+\asc(x_1,\dots , x_{i-1})]$ for all $2\leq i \leq n$. For
instance, (0, 1, 0, 2, 3, 1, 0, 0, 2) is an ascent sequence.
The length (number of entries) of a sequence $x$ is denoted $|x|$.

Let $\sym_n$ be the symmetric group on $n$ elements.  Let
$V=\{v_1,v_2,\dots,v_n\}$ with $v_1<v_2<\dots<v_n$ be any finite
subset of $\NN$. The \emph{standardisation} of a permutation $\pi$ on
$V$ is the permutation $\st(\pi)$ on $[n]
:=\{1, 2, \dots, n\}$ obtained from $\pi$ by
replacing the letter $v_i$ with the letter $i$. As an example,
$\st(19452) = 15342$. Let $\R_{n}$ be 
the following set of permutations:
\begin{equation*}
\R_{n} = \{\,\pi_1\dots\pi_n \in \sym_n : 
\text{ if $\st(\pi_i\pi_j\pi_k)=231$ then $j\neq i+1$ or $\pi_i\neq \pi_k+1$}
\,\}.
\end{equation*}
Equivalently, if $\pi_i\pi_{i+1}$ forms an ascent, then $\pi_i-1$ is
not found to the right of this ascent.
This class of permutations could be more descriptively written as
$\R_{n} = \sym_n\!\left(\pattern\right)$,
the set of permutations \emm avoiding, the pattern in the
diagram. 
Dark lines indicate adjacent entries (horizontally or vertically),
whereas lighter lines indicate an elastic distance between the entries.
Conversely, $\pi$ \emm contains, this pattern if there exists $i<k$ such
that $\pi_k+1=\pi_i<\pi_{i+1}$.  As
illustrated below, the permutation 31524 avoids the pattern 
$\pattern$ while the permutation 32541 
contains it.
$$
\begin{minipage}[c]{16em}\scalebox{0.6}{\includegraphics{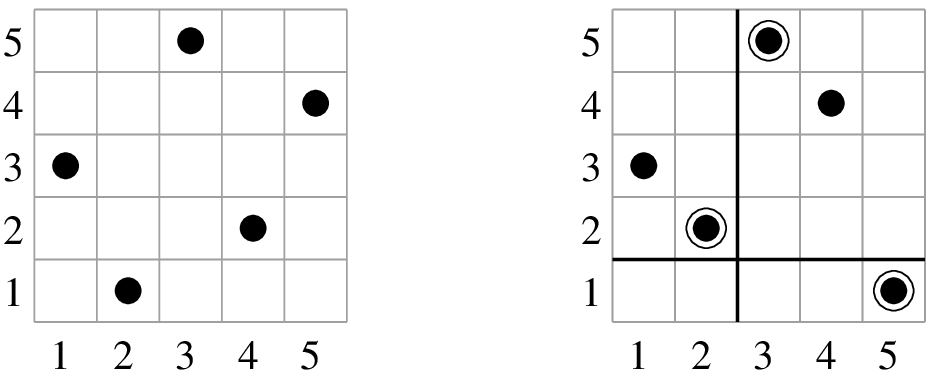}}
\end{minipage}
$$ 

Clearly, this  example can be generalized to
any pattern consisting of a permutation plus some dark (vertical and
horizontal) lines.  Vertical   lines represent a constraint of
adjacency of the \emm positions,, while horizontal lines represent a
constraint of adjacency of the \emm values,. When there is no dark
line, we recover the  standard notion of containment of
a \p. When only  vertical lines are allowed, that is, constraints on
the positions, we recover the \emm vincular, (or \emm generalized,)
patterns of  Babson and Steingr\'imsson~\cite{BABSON_EINAR}. For
symmetry reasons, it seems natural to allow constraints on values as
well, and this  is precisely  what our \emm bivincular patterns,,
defined formally below, achieve.

Let us  now give a formal 
definition of bivincular
patterns. This is not needed for the rest of this paper, and the
reader may, without loss of continuity, skip 
the next three paragraphs.
We define a \emph{bivincular permutation} (or {\emph{bivincular pattern}}) to be
a triple
$p=(\sigma,X,Y)$, where $\sigma$ is a permutation on $[k]$ and $X$ and
$Y$ are subsets of $[0,k]$. An occurrence of $p$ in a permutation
$\pi=\pi_1\dots\pi_n$ on $[n]$ is subsequence $o =
\pi_{i_1}\dots\pi_{i_k}$ such that $\st(o)=\sigma$ and
$$
\forall x\in X,\, i_{x+1} = i_x+1
\quad\text{and}\quad
\forall y\in Y,\, j_{y+1} = j_y+1,
$$ where $\{\pi_{i_1},\dots,\pi_{i_k}\}=\{j_1,\dots, j_k\}$ and
$j_1<j_2<\dots<j_k$; by convention, $i_0=j_0=0$ and
$i_{k+1}=j_{k+1}=n+1$. With this definition we have $\R_n =
\sym_n\big((231,\{1\},\{1\})\big)$. Note also that the number of
bivincular permutations of length $n$ is $4^{n+1}n!$.

The classical patterns are those of the form
$p=(\sigma,\emptyset,\emptyset)$. Vincular patterns are of the form
$p=(\sigma,X,\emptyset)$. Let $p=(\sigma,X_p,Y_p)$ and
$q=(\tau,X_q,Y_q)$ be any two patterns.  If $\si$ and $\tau$ have the
same length, we define their composition, or product, by $p\circ q =
(\,\sigma\circ\tau,\, X_p\Delta Y_q,\, Y_p\Delta X_q\,)$, where
$A\Delta B = (A-B) \cup (B-A)$ is the symmetric difference.  This
operation is not associative, but it admits a right identity,
$(\mathrm{id},\emptyset,\emptyset)$, and every element
$p=(\sigma,X,Y)$ has an inverse $p^{-1} = (\sigma^{-1},Y,X)$; this
turns the set of bivincular permutations of length $n$ into a
quasigroup with right identity. Also, reverse is defined by
$p^r=(\sigma^r,n+1-X,Y)$ and complement is defined by $p^c =
(\sigma^c,X,n+1-Y)$, in which $k-A$ denotes the set $\{k-a : a\in
A\}$. 
Thus the set of bivincular patterns has the full symmetry of a square.

One simple instance of bivincular pattern avoidance that has already
appeared in the literature is the set of \emm irreducible,
permutations~\cite{albert}, that is, permutations such that
$\pi_{i+1} \not = \pi_i -1$ for all $i$. With our terminology, these
are the permutations avoiding $(21, \{1\},\{1\})$. Similarly, the \emm
strongly irreducible, permutations of
\cite{atkinson} are the $(21, \{1\},\{1\})$- and $(12, \{1\},\{1\})$-avoiding 
permutations.

Let us now return to the set $\R:=\cup_n \R_{n}$ of permutations
avoiding $(231, \{1\},\{1\})$.  Let $\pi$ be a permutation of
$\R_{n}$, with $n>0$. Let $\tau$ be obtained by deleting the entry $n$
from $\pi$. Then $\tau \in \R_{n-1}$. Indeed, if
$\tau_i \tau_{i+1} \tau_j$ is an occurrence of the forbidden pattern
in $\tau$ (but not in $\pi$), then this implies that
$\pi_{i+1}=n$. But then $\pi_i \pi_{i+1} \pi_{j+1}$ would form an
occurrence of the forbidden pattern in $\pi$.

This property allows us to construct the permutations of $\R_{n}$ inductively,
starting from the empty permutation and adding a new maximal value at each
step.  
(This is  the \emm generating tree, approach, systematized by West~\cite{west-trees}.)
Given $\tau = \tau_1 \dots \tau_{n-1} \in \R_{n-1}$, the sites
where $n$ can be inserted in $\tau$ so as to produce an element of
$\R_{n}$ are called \emm active,.  It is easily seen that the site
before $\tau_1$ and the site after $\tau_{n-1}$ are always active.
The site between the entries $\tau_i$ and $\tau_{i+1}$ is {active} if
and only if $\tau_i=1$ or $\tau_i - 1$ is to the left of $\tau_i$.
Label the active sites, from left to right, with labels 0, 1, 2 and so
on.  Observe that the site immediately to the left of the maximal
entry of $\tau$ is always active.

Our bijection $\Lambda$ between permutations of  $\R_{n}$ and ascent
sequences of length  $n$  
is defined recursively on $n$ as follows. For $n=1$, we set
$\Lambda(1)=(0)$.  Now let $n \ge 2$, and suppose that $\pi \in
\R_{n}$ is obtained by inserting $n$ in the active site labeled $i$ of
a permutation $\tau \in \R_{n-1}$. Then the sequence associated with $\pi$ is
$\Lambda(\pi):=(x_1, \dots, x_{n-1}, i)$, where $(x_1,
\dots,x_{n-1})=\Lambda(\tau)$.

\begin{example}
  The permutation  $\pi = 6 1 8 3 2 5 4 7$ corresponds to the sequence
  $x=(0,1,1,2,2,0,3,1)$,   since it is obtained by the following
  insertions (the subscripts indicate the labels of the active
  sites):
  \begin{align*}
    _0 1 _1 
    &\,\xmapsto{x_2=1}\,  {_0} 1 {_1} 2 {_2} \\
    &\,\xmapsto{x_3=1}\,  {_0} 1 {_1} 3\nyip 2 _2 \\
    &\,\xmapsto{x_4=2}\,  {_0} 1 {_1} 3\nyip 2 {_2} 4 {_3} \\
    &\,\xmapsto{x_5=2}\,  {_0} 1 {_1} 3\nyip 2 {_2} 5\nyip 4 {_3}  \\
    &\,\xmapsto{x_6=0}\,  {_0} 6\nyip 1 {_1} 3\nyip 2 {_2} 5\nyip 4 {_3} \\
    &\,\xmapsto{x_7=3}\,  {_0} 6\nyip 1 {_1} 3\nyip 2 {_2} 5\nyip 4 {_3} 7 {_4} \\
    &\,\xmapsto{x_8=1}\, \nyip 6\nyip 1\nyip 8\nyip 3\nyip 2\nyip 5\nyip 4\nyip 7.
  \end{align*}
\end{example}

\begin{theorem}\label{th:lambda}
\label{av-asc}
  The map $\Lambda$ is a bijection from  $\R_{n}$ to the set of  ascent
  sequences of length $n$.
\end{theorem}
\begin{proof}
Since the sequence $\Lambda(\pi)$ encodes the construction of $\pi$,
the map $\Lambda$ is injective. We want to prove that the image of
$\R_n$ is the set $\A_n$ of ascent sequences of length $n$.  Let
$s(\pi)$ denote the number of active sites of the permutation $\pi$. Our
recursive description of the map $\Lambda$ tells us that $x=(x_1,
\dots,x_n)\in \Lambda(\R_n)$ if and only if
\begin{equation}\label{description}
x'=(x_1, \dots, x_{n-1})\in \Lambda(\R_{n-1}) \quad 
\text{ and } \quad 0\le x_n \le s\left(\Lambda^{-1}( x')\right)-1
\end{equation}
(recall that the leftmost active site is labeled 0, so that the
rightmost one is  $s(\pi)-1$).

We will prove by induction on $n$ that for all $\pi \in \R_n$, with
associated sequence $\Lambda(\pi)=x=(x_1, \dots, x_n)$, one has
\begin{equation}\label{properties}
s(\pi)=2+\asc(x)  \quad 
\text{ and } \quad b(\pi)= x_n,
\end{equation}
where $b(\pi)$ is the label of the site located just before the
maximal entry of $\pi$. Clearly, this will convert the above
description~\eqref{description} of $ \Lambda(\R_n)$ into the
definition of ascent sequences, thus concluding the proof.

So let us focus on the properties~\eqref{properties}. They obviously
hold for $n=1$. Now assume they hold for some $n-1$, with $n\ge 2$,
and let $\pi \in \R_n$ be obtained by inserting $n$ in the active site
labeled $i$ of $\tau \in \R_{n-1}$. Then $\Lambda(\pi)= x=(x_1,
\dots, x_{n-1},i)$ where $\Lambda(\tau)= x'=(x_1, \dots,
x_{n-1})$. Every entry of $\pi$ smaller than $n$
is followed in $\pi$ by an active site if and only if it was followed in $\tau$ by
an active site.  The leftmost site also remains active. Consequently,
the label of the active site preceding $n$ in $\pi$ is $i=x_n$, which
proves the second property. Thus,
in order to determine $s(\pi)$,
 the only question is whether the site
following $n$ is active in $\pi$.  There are two cases to
consider. Recall that, by the induction hypothesis,
$s(\tau)=2+\asc(x')$ and $b(\tau)=x_{n-1}$.
  
{\sf Case 1:}\, If $0\leq i\leq b(\tau)=x_{n-1}$ then
$\asc(x)=\asc(x')$ and the entry $n$ in
$\pi$ is to the left of $n-1$. So the number of active sites remains
unchanged: $s(\pi) = s(\tau)=2+\asc(x')=2+\asc(x)$.
    
{\sf Case 2:}\, If $i>b(\tau)=x_{n-1}$ then $\asc(x)=1+\asc(x')$ and
the entry $n$ in $\pi$ is to the right of $n-1$.  The site that
follows $n$ is thus active, and $s(\pi ) =
1+s(\tau)=3+\asc(x')=2+\asc(x)$.  This concludes the proof.
\end{proof}

%%%%%%%%%%%%%%%%%%%%%%%%%%%%%%%%%%%%%%%%%%%%%%%%%%%%%%%%%%%%%%%%
\section{Ascent sequences and unlabeled {\tpt}-free posets}
\label{sec:poset}
%%%%%%%%%%%%%%%%%%%%%%%%%%%%%%%%%%%%%%%%%%%%%%%%%%%%%%%%%%%%%%%%
Let $\P_n$ be the set of unlabeled {\tpt}-free posets on $n$ elements.
In this section we shall give a bijection between $\P_n$ and the set
$\A_n$ of ascent sequences of length $n$.
As in the previous section, this bijection encodes a recursive way
of constructing \tpt-free posets by adding one new (maximal)
element. There is of course  a corresponding removal operation, but it
 is less elementary than in the case of permutations.
Before giving these operations we need to define some terminology.

Let $D(x)$ be the set of \emm predecessors, of $x$ (the strict down-set of
$x$). Formally,
$$
D(x)=\{\, y : y<x \,\}.
$$ 
It is well-known---see for example
Bogart~\cite{bogart}---that a poset is {\tpt}-free if and only if
its sets of predecessors, $\{D(x) : x\in P\}$, can be linearly ordered
by inclusion. For completeness we prove this result here. 
\begin{lemma}
A poset $P$ is {\tpt}-free
if and only if
the set of strict downsets of $P$ can be linearly ordered by inclusion.
\end{lemma}

\begin{proof}
If the set of strict downsets of $P$ cannot be linearly ordered by
inclusion, then there are two incomparable elements $x,y \in P$ such
that both $D(x)\setminus D(y)$ and $D(y)\setminus D(x)$ are non-empty.
Let $x'\in D(x)\setminus D(y)$ and $y' \in D(y) \setminus D(x)$. Then
the induced subposet on the elements $\{x,x',y,y'\}$ is isomorphic to
{\tpt}. Conversely, if $P$ contains an induced subposet $\{x>x',
y>y'\}$ isomorphic to {\tpt}, then $D(x)$ and $D(y)$ are such that
both $D(x)\setminus D(y)$ and $D(y)\setminus D(x)$ are non-empty.
\end{proof}
%%%%%%%%%%%%%%%%%%%%%%%%%%%%%%%%%%%%%%%%%%%%%
Let
$$
 D(P) = \{D_0,D_1,\dots,D_{k}\}
$$ 
with $\emptyset= D_0\subset D_1\subset\dots\subset D_{k}$.
In this context we define $D_i(P)=D_i$ and we write $\rank(P) = k$.
We say the element $x$ is at level
$i$ in $P$ if $D(x)=D_i$ and we write
$\rank(x)=i$ .  The set of all elements at 
level  $i$ we denote $L_i(P)
= \{\, x\in P : \rank(x)=i \,\}
= \{\, x\in P : D(x)=D_i \,\}$. 
For instance, $L_0(P)$ is the set of minimal elements.
All the elements of $L_{k}(P)$ are maximal, but there may be maximal
elements of $P$ at level less than $k$. 
If $L_i(P)$ contains a maximal element, we say that \emm the
level, $i$ contains a maximal element. 
Let $\srank(P)$ be the minimum level containing a maximal element.

\begin{example}
Consider the following {\tpt}-free poset $P$,
which we have labeled for convenience:
$$
\begin{minipage}{5.2em}
  \includegraphics[scale=0.6]{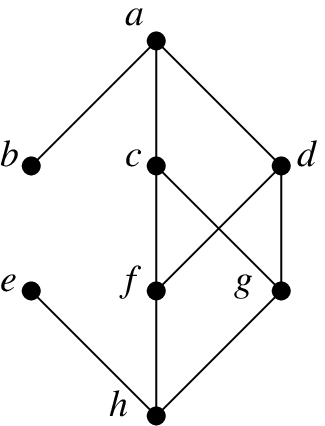}
\end{minipage}
\quad = 
\quad
\begin{minipage}{4em}
  \includegraphics[scale=0.6]{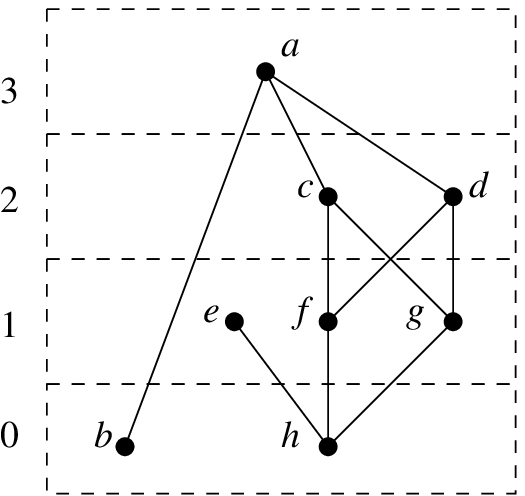}
\end{minipage}
\qquad\quad
$$ 
The diagram on the right shows the poset redrawn according to the
levels of the elements.
 We have $D(a)
=\{b,c,d,f,g,h\}$, $D(b) =\emptyset$, $D(c) = D(d) = \{f,g,h\}$, $D(e)
= D(f) = D(g) =\{h\}$ and $D(h) =\emptyset$. These may be ordered by
inclusion as
$$
\begin{array}{ccccccc} 
  \underbrace{D(h) = D(b)} &\!\!\!\subset\!\!\!&
  \underbrace{D(e) = D(f) = D(g)} &\!\!\!\subset\!\!\!& 
  \underbrace{D(c) = D(d)} &\!\!\!\subset\!\!\!& 
  \underbrace{D(a)}. \\[1.8ex] 
  \rank(h)=\rank(b)=0 &&
  \rank(e)=\rank(f)=\rank(g)=1 && 
  \rank(c)=\rank(d)=2 && 
  \rank(a)=3
\end{array}
$$ 
Thus $\rank(P)=3$.
 The maximal elements of $P$ are $e$ and
$a$, and they lie respectively at levels 3 and 1. 
Thus $\srank(P)=1$. In addition, $D_0=\emptyset$, $D_1=\{h\}$,
$D_2=\{f,g,h\}$ and $D_3=\{b,c,d,f,g,h\}$. With $L_i = L_i(P)$ we also
have $L_0=\{h,b\}$, $L_1=\{e,f,g\}$, $L_2=\{c,d\}$ and $L_3=\{a\}$.
\end{example}

%%%%%%%%%%%%%%%%%%%%%%%%%%%%%%%%%%%%%%%%%%%%%%%%%%%%%%%%%%%%%%%%%%%%%
\subsection{Removing an element from a \tpt-free poset}\label{sec:rem}
%%%%%%%%%%%%%%%%%%%%%%%%%%%%%%%%%%%%%%%%%%%%%%%%%%%%%%%%%%%%%%%%%%%%%

Let us begin with the removal operation, which will be the counterpart
of the deletion of the last entry in an ascent sequence (or the
deletion of the largest entry in a permutation of $\R$). 
Let $P$ be a \tpt-free poset of cardinality $n\ge 2$, and let
$i=\srank(P)$ be the minimum level of $P$ containing a maximal
element. All the maximal elements located at level $i$ are
order-equivalent in the unlabeled poset $P$. We will 
remove  one of them.
Let $Q$ be the poset that results from applying:
\begin{enumerate}
\item[(\subi)] If $|L_i(P)|>1$
  then simply remove one of the maximal elements at level $i$. 
 
\item[(\subii)] If  $|L_i(P)|=1$ and $i=\rank(P)$ then remove the unique
element lying at level $i$.

\item[(\subiii)] If  $|L_i(P)|=1$ and  $i<\rank(P)$ 
then set $\N=D_{i+1}(P)\setminus D_{i}(P)$.
  Make each element in
  $\N$ a maximal element of the poset by deleting
from the order all relations $x < y$ where $x \in \N$.
Finally, remove the unique element lying at level $i$.
\end{enumerate}

\begin{example}\label{ex:rem}
Let $P$ be the unlabeled {\tpt}-free poset with this Hasse 
diagram:
$$
\begin{minipage}{8.6em}
  \includegraphics[scale=0.85]{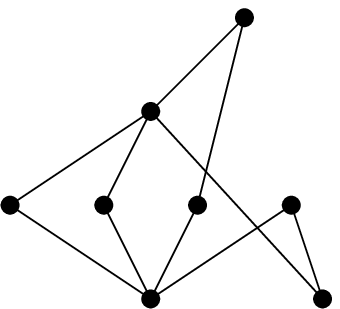}
\end{minipage} =\quad 
\begin{minipage}{10em}
  \scalebox{0.70}{\includegraphics{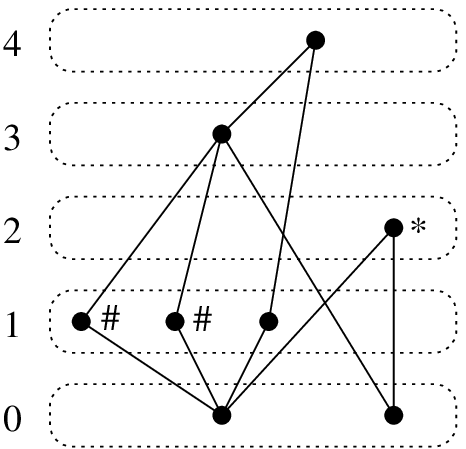}}
\end{minipage}
$$
The diagram on the right shows the poset redrawn according to the
levels of the elements.  There is a unique maximal element of minimal
level, which is marked with $*$ and lies at level 2, so that
$\srank(P)=2$.  Since there is a unique element at level $2<\rank(P)$,
apply {\subiii} to remove it.  The elements of $\N$ are indicated
by \#'s.  In order to delete all relations of the form $x<y$ where
$x\in \N$, one deletes from the Hasse diagram all edges corresponding
to coverings of elements of $\N$, and adds an edge between the
elements at level 0 and 3 to preserve their relation. Finally, one
removes the element at level 2.
This gives a new \tpt-free poset, with
level numbers shown on the left.
$$\mapsto\quad
\begin{minipage}{10em}\scalebox{0.70}
  {\includegraphics{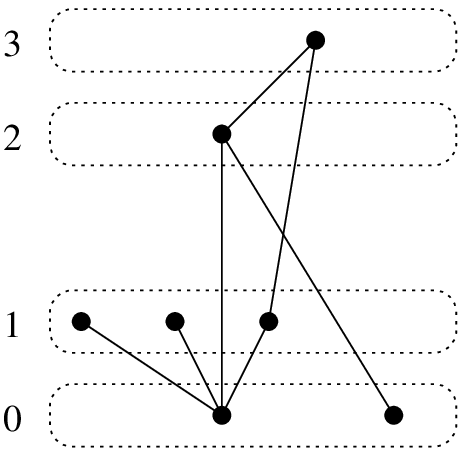}}
\end{minipage} =\quad
\begin{minipage}{10em}
  \scalebox{0.70}{\includegraphics{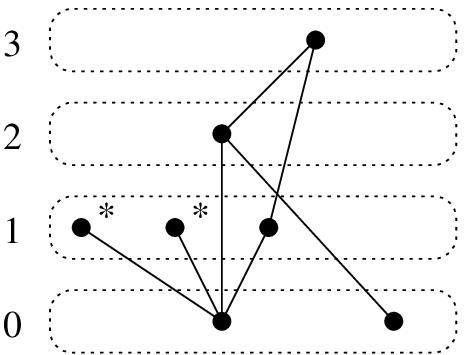}}
\end{minipage}
$$
There are now two maximal elements of minimal level $\srank=1$, both marked
 by $*$. Remove one of them according to rule {\subi}.
This gives the poset shown on the left below, for which $\srank$ is still
1. Apply {\subi} again to obtain the  poset on the right.
$$
\mapsto\quad
\begin{minipage}{8.5em}
  \scalebox{0.70}{\includegraphics{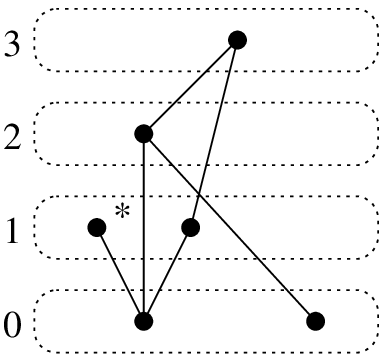}}
\end{minipage}
\mapsto \quad
\begin{minipage}{7.4em}
  \scalebox{0.70}{\includegraphics{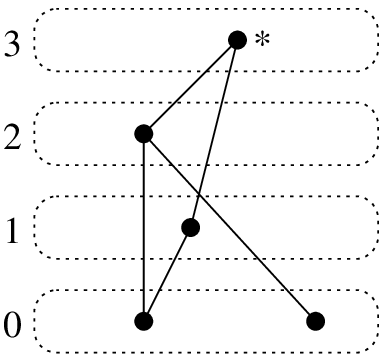}}
\end{minipage}
$$
There is now a single maximal element, lying at maximal level 3,
so we apply rule {\subii}:
$$
\mapsto\quad
\begin{minipage}{10em}
  \scalebox{0.70}{\includegraphics{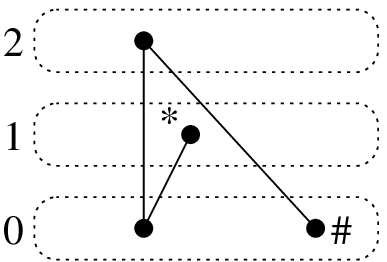}}
\end{minipage}
$$
The maximal element of minimal level
is now alone on level $\srank(P)=1<\rank(P)$ so
 apply {\subiii}. The set  $\N$ consists of the rightmost point at
level 0, giving
$$
\mapsto\quad
\begin{minipage}{9.5em}
  \scalebox{0.70}{\includegraphics{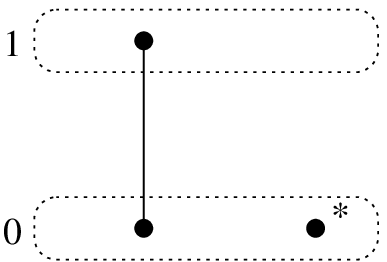}}
\end{minipage}
$$
The maximal element of minimal level
 is not alone at level  0, so  apply {\subi}:
$$
\mapsto\quad
\begin{minipage}{6.3em}
  \scalebox{0.70}{\includegraphics{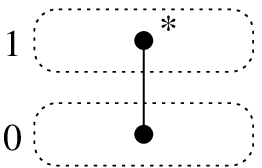}}
\end{minipage} \mapsto\;\;\;
\begin{minipage}{4.5em}
  \scalebox{0.70}{\includegraphics{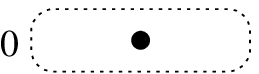}}
\end{minipage}
$$
We have thus reduced the original poset $P$ to a one element poset by
 removing the elements in a canonical order.
\end{example}

Let us now check that the removal operation gives a \tpt-free poset,
and establish some elementary properties of this operation. If
$\srank(P)=i$, and the removal operation, applied to $P$, gives $Q$,
we define  $\remove(P)=(Q,i)$. 

\begin{lemma}\label{remove}
  If $n\geq 2$, $P\in\P_n$ and $\remove(P)=(Q,i)$, then $Q\in\P_{n-1}$
  and $0\leq i\leq 1+\rank(Q)$. Also,
  $$
  \rank(Q) = 
  \begin{cases}
    \rank(P)   & \text{if } i \leq \srank(Q),\\
    \rank(P)-1 & \text{if } i > \srank(Q).
  \end{cases}
  $$
\end{lemma}
\begin{proof}
We examine separately the 3 cases described above.

If $|L_i(P)|>1$ then one simply removes a maximal element at level $i$
  to obtain $Q$: the set of sets of predecessors is unchanged, and remains
linearly   ordered. Hence $Q\in \R_{n-1}$. Also, $\rank(Q)
=\rank(P)$. 
The maximal elements of $Q$ were already maximal in $P$. Thus the
maximal elements of lowest level in $Q$ are at
level $i$ at least, that is, $\srank(Q) \ge i$. 
  
If $|L_i(P)|=1$ and $i=\rank(P)$, one removes the unique element of
maximal level. One has now $D(Q)=D(P)\setminus \{D_i(P)\}$, which is
still linearly ordered. Also, $\rank(Q) =\rank(P)-1$. In particular,
$i=\rank(Q)+1>\srank(Q)$.

Finally, if $|L_i(P)|=1$  and $i<\rank(P)$, define the set $\N$ as
in \subiii. 
By construction, the set of sets of predecessors of $Q$ is%definition,
$$
D(Q)=
\big\{\,
  D_0(P),\,\dots ,\, D_{i-1}(P),\, 
  D_{i+1}(P) \setminus \N,\,\dots ,\, D_{\rank(P)}(P) \setminus \N
\,\big\}.
  $$
To prove that $D(Q)$ can be linearly ordered, it suffices to prove
 that  $D_{i-1}(P)\subset D_{i+1}(P) \setminus \N$.
By definition, $\N=D_{i+1}(P)\setminus D_{i}(P)$ and hence
  \begin{align*}
  D_{i+1}(P)\setminus \N 
  &= D_{i+1}(P) \setminus  \big(D_{i+1}(P)\setminus D_{i}(P)\big) \\
  &= D_{i+1}(P) \cap D_{i}(P) \\
  &= D_{i}(P) \\
  &\supset D_{i-1}(P).
  \end{align*}
  It is also clear that $\rank(Q)=\rank(P)-1$.  
The elements of $\N$ are maximal in $Q$ and lie at level $<i$. Hence
 $\srank(Q) < i$. 
\end{proof}

%%%%%%%%%%%%%%%%%%%%%%%%%%%%%%%%%%%%%%%%%%%%%%%%%%%%%%%%%%%%%%%%%%%%%%
\subsection{Adding an element to a \tpt-free poset}
%%%%%%%%%%%%%%%%%%%%%%%%%%%%%%%%%%%%%%%%%%%%%%%%%%%%%%%%%%%%%%%%%%%%%%

Let us now define the addition operation, which adds a maximal element
to a \tpt-free poset $Q$.

Given $Q \in \P_{n-1}$ and $0\leq i \leq 1+\rank(Q)$, let $\add(Q,i)$ be
the poset $P$ obtained from $Q$ according to the following:
\begin{enumerate}
\item[(\addi)] 
If $ i \leq \srank(Q)$ then introduce a new maximal element 
which covers the same elements as the elements of $L_i(Q)$.

\item[(\addii)] If $i=1+\rank(Q)$, add a new element covering all
maximal elements of $Q$.

\item[(\addiii)] If $\srank(Q)<i\leq\rank(Q)$, add a new element 
covering the same elements  as the elements of $L_i(Q)$.
Let  $\M$ be the set of maximal elements 
of $Q$ of level less than $i$.  
Add all relations $x\le y$ where 
$x\leq z$ for some $z \in \M$  and $y \in L_i(Q)\cup \cdots
\cup L_{\rank(Q)}(Q)$.  In particular,  every element of $\M$ is now
covered by every minimal element of the poset induced by $L_i(Q)\cup \cdots
\cup L_{\rank(Q)}(Q)$. 

\end{enumerate}

\begin{example}\label{ex:add}
Starting from the one-element poset,
we add successively 7 points according to the rules above, where the
parameter $i$ takes the following values: $i= 1,2,3,1,0,1,2$. Note
that the sequence  $(0,1,2,3,1,0,1,2)$ is an ascent sequence. This is
of course not an accident.
For each step, the new element is circled. 
$$
\begin{minipage}{2.25em}
  \quad\scalebox{0.70}{\includegraphics{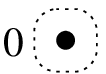}}
\end{minipage}
\;\; \xmapsto[\addii]{i_2=1}\;\; 
\begin{minipage}{2.25em}
  \quad\scalebox{0.70}{\includegraphics{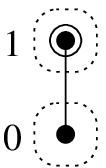}}
\end{minipage}
\;\; \xmapsto[\addii]{i_3=2}\;\; 
\begin{minipage}{2.25em}
  \quad\scalebox{0.70}{\includegraphics{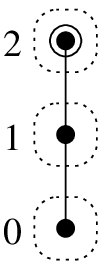}} 
\end{minipage} 
\;\; \xmapsto[\addii]{i_4=3}\;\; 
\begin{minipage}{2.25em}
  \quad\scalebox{0.70}{\includegraphics{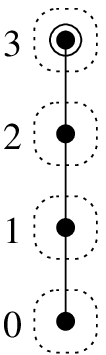}} 
\end{minipage} 
\;\; \xmapsto[\addi ]{i_5=1}\;\, 
\begin{minipage}{3em}
  \quad\scalebox{0.70}{\includegraphics{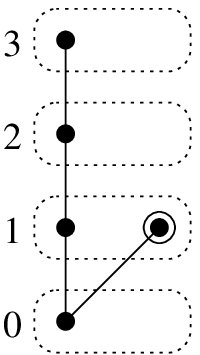}} 
\end{minipage} 
$$
$$
\;\; \xmapsto[\addi ]{i_6=0} \;\; 
\begin{minipage}{6.8em} 
  \scalebox{0.7}{\includegraphics{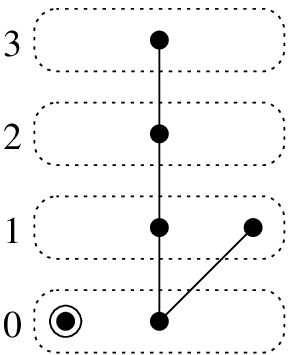}} 
\end{minipage} 
\;\; \xmapsto[\addiii]{i_7=1} \;\; 
\begin{minipage}{6.8em} 
  \scalebox{0.7}{\includegraphics{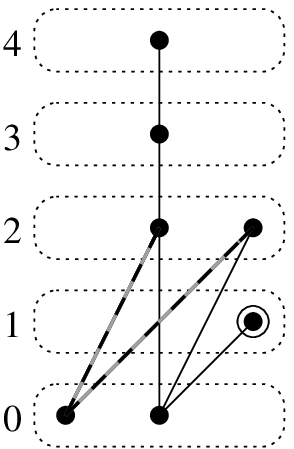}} 
\end{minipage} 
\;\; \xmapsto[\addiii]{i_8=2} \;\; 
\begin{minipage}{6.8em}
  \quad \scalebox{0.7}{\includegraphics{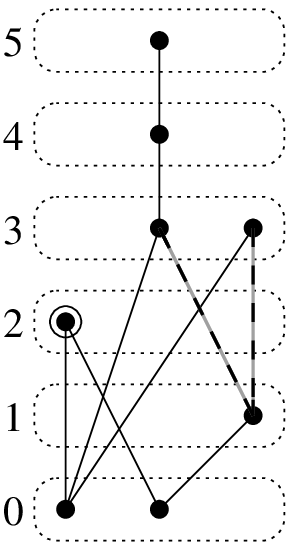}} 
\end{minipage}\smallskip
$$
In the final two steps, where the operation \addiii\ is used, 
 we have inserted dashed lines
indicating the covering of  the elements of $\M$.
Observe
 that in a last step, the addition of these new coverings makes
two edges  of the next-to-last diagram transitive: they do not appear
any more in the final diagram.
\end{example}

Let us now check that the addition operation gives a \tpt-free poset,
and establish some elementary properties of this operation. 

\begin{lemma}\label{myadd}
  If $n\geq 2$, $Q\in\P_{n-1}$, $0\leq i\leq1+\rank(Q)$ and
  $P=\add(Q,i)$, then $P \in\P_{n}$. Also,
  $$
  \srank(P) =i \quad \text{ and } \quad 
  \rank(P) =
  \begin{cases}
    \rank(Q)   & \text{if } i \leq \srank(Q), \\
    \rank(Q)+1 & \text{if } i > \srank(Q).
  \end{cases}
  $$
\end{lemma}

\begin{proof}
   We examine separately the 3 cases described above.
  
  If $ i \leq \srank(Q)$, then \addi\ is used.  We want to show that the
  set $D(P)=\{D(x):x\in P\}$ of sets of predecessors can be linearly
  ordered. This is however trivial: By definition of \addi\ we have
  $D(P) = D(Q)$ which is linearly ordered.  The set of predecessors of
  the new element is $D_i(Q)$, so it lies at level $i$. As this element
  is maximal, and its level $i$ is not larger than $\srank(Q)$, we have
  $\srank(P)=i$.  Finally, it follows from $D(P) = D(Q)$ that $\rank(P)
  =\rank(Q)$.
  
  If $i=1+\rank(Q)$, then  \addii\ is used. 
  The set $D(P)$ is $D(Q) \cup \{Q\}$, which is still linearly ordered
  by inclusion. The highest level increases by one:
  $\rank(P)=\rank(Q)+1$. Finally, the new element is the only maximal
  element of $P$, so that $\srank(P)=\rank(P)=1+\rank(Q)=i$.
  
  If $\srank(Q)<i\leq\rank(Q)$, then \addiii\ is used. The new element
  has set of predecessors $D_i(Q)$. The elements that had level $i$
  or more in $Q$ now include the elements of $\M$ among their
  predecessors. Consequently,
  \begin{equation}\label{DP}
    D(P)= 
    \big\{\,D_0(Q),\,\dots,\, D_i(Q),\, 
    D_i(Q)\cup \M ,\, D_{i+1}(Q) \cup \M ,\,\dots ,\, 
    D_{\rank(Q)} (Q) \cup \M\,\big\},
  \end{equation}
  which is linearly ordered. From this expression for $D(P)$ we
  also see that $\rank(P)=\rank(Q)+1$, as claimed. Moreover, as all
  elements of level less than $i$ in $Q$ are now covered, the new
  element is the only maximal element of minimal level, so that
  $\srank(P)=i$. 
\end{proof}

Let us now prove the compatibility of our removal and addition operations.
\begin{lemma}\label{lemma:add-rm}
  For any {\tpt}-free poset $Q$ and integer $i$ such that $0\leq i\leq
  1+\rank(Q)$ we have $\remove(\add(Q,i)) = (Q,i)$. And if $Q$ has
  more than one element we also have $\add(\remove(Q)) = Q$.
\end{lemma}

\begin{proof}
  Let us begin with the first statement, and denote $P=\add  (Q,i)$. 
  Recall that $\srank(P)=i$, so that the removal operation applied to
  $P$  takes out an element of level $i$ and gives $\psi(P)= (R,i)$. We
  want to prove that $R=Q$. 
  
  Assume that $i \leq \srank(Q)$ so that \addi\ is used
to construct $P$ from $Q$. 
The new  element  is introduced at level $i$ and is  not  alone at this
  level. Thus the removal operation \subi\ is applied to $P$, and simply
  removes one maximal element at level $i$---either the one that was added,
  or another, order-equivalent, one. Thus $Q$ and $R$ coincide, as
  unlabeled posets.  

  Assume that $i=1+\rank(Q)$ so that \addii\ is used. The new element
  is the only maximal element in $P$, so that the removal operation
  \subii\ is applied to $P$, and simply removes this maximal
  element. Thus again, $R=Q$.

  Assume that $\srank(Q)<i\leq\rank(Q)$ so that \addiii\ is used. The
  new element is maximal, and is the only element at level
  $i<\rank(P)=1+\rank(Q)$. Thus it will be removed using \subiii.  Let
  $\M$ be the set of maximal elements of $Q$ of level less than
  $i$. The set $\N$ that occurs in the description of \subiii\ is
  $D_{i+1}(P) \setminus D_i(P)$. According to~\eqref{DP}, this set
  coincides with $\M$. Hence the covering relations that were added to
  go from $Q$ to $P$ are now destroyed when going from $P$ to
  $R$. Thus $R=Q$.

  A similar argument (with the two transformations interchanged) gives
  the second statement of the lemma.
\end{proof}

%%%%%%%%%%%%%%%%%%%%%%%%%%%%%%%%%%%%%%%%%%%%%%%%%%%%%%%%%
\subsection{From \tpt-free posets to ascent sequences}
%%%%%%%%%%%%%%%%%%%%%%%%%%%%%%%%%%%%%%%%%%%%%%%%%%%%%%%%%
Our bijection $\Psi$ between \tpt-free posets of cardinality $n$ 
and ascent sequences of length  $n$ 
is defined recursively on $n$ as follows. For $n=1$, we associate with
the one-element poset the sequence $(0)$.
Now let $n \ge 2$,
and suppose that the removal operation, applied to $P \in \P_{n}$,
gives $\psi(P)=(Q,i)$. In other words, $P$ is obtained from $Q$ by
adding a new maximal element at level $i$, following our addition
procedure. Then the sequence
associated with 
$P$ is $\Psi(P):=(x_1, \dots,  x_{n-1}, i)$, where
$(x_1, \dots,x_{n-1})=\Psi(Q)$.

For instance, the poset of Example~\ref{ex:rem} corresponds to the sequence
$(0,1,0,1,3,1,1,2)$, while the poset  of Example~\ref{ex:add} corresponds to
the sequence  $(0,1,2,3,1,0,1,2)$.
\begin{theorem}\label{poset-asc}
  The map $\Psi$ is a one-to-one correspondence between {\tpt}-free
  posets of size $n$ and ascent sequences of length $n$.
\end{theorem}

\begin{proof}
  Since the sequence $\Psi(P)$ encodes the construction of the poset $P$, the map
  $\Psi$ is injective. We want to prove that the image of $\P_n$ is the 
  set $\A_n$ of ascent sequences of length $n$.   
  Our recursive description of the map
  $\Psi$ tells us  that $x=(x_1, \dots,x_n)\in \Psi(\P_n)$ if and only if
  \begin{equation}\label{description-bis}
  x'=(x_1, \dots, x_{n-1})\in \Psi(\P_{n-1}) \quad 
  \text{ and } \quad 0\le x_n \le 1+ \rank\left(\Psi^{-1}( x')\right).
  \end{equation}
  We will prove by induction on $n$ that for all $P \in \P_n$, with
  associated sequence $\Psi(P)=x=(x_1, \dots, x_n)$, one has
  \begin{equation}\label{properties-bis}
  \rank(P)=\asc(x)  \quad 
  \text{ and } \quad \srank(P)= x_n.
  \end{equation}
  Clearly, this will convert the above
  description~\eqref{description-bis} of $ \Psi(\P_n)$ into 
  the definition of ascent sequences, thus concluding the proof.
  
  So let us focus on the properties~\eqref{properties-bis}. They
  obviously hold for $n=1$. Now assume they hold for some $n-1$, with
  $n\ge 2$, and let $P \in \P_n$ be obtained by adding a new element
  at level $i$ in $Q \in \P_{n-1}$. Then $\Psi(P)= x=(x_1, \dots,
  x_{n-1},i)$ where $\Psi(Q)= x'=(x_1, \dots, x_{n-1})$.  By the
  induction hypothesis, $\rank(Q)=\asc(x')$ and
  $\srank(Q)=x_{n-1}$. Lemma~\ref{myadd} gives $\srank(P)=i$ and
  $$
  \rank(P) = 
  \begin{cases}
    \asc(x')   & \text{if } i \leq x_{n-1}, \\
    \asc(x')+1 & \text{if } i>  x_{n-1}. 
  \end{cases}
  $$
  The result follows.
\end{proof}

%%%%%%%%%%%%%%%%%%%%%%%%%%%%%%%%%%%%%%%%%%%%%%%%%%%%%%%%%%%%%%%%
\section{Modified ascent sequences and their applications}\label{sec:mod}
%%%%%%%%%%%%%%%%%%%%%%%%%%%%%%%%%%%%%%%%%%%%%%%%%%%%%%%%%%%%%%%%

In this section we introduce a transformation on ascent sequences and
show some applications. For instance, this transformation can be used
to give a  non-recursive description of the bijection $\Lambda$
between permutations of $\R$ and ascent sequences. It is also useful
to characterize the image by $\Lambda$ of a subclass of $\R$ studied
by Pudwell~\cite{lara}, which we
 enumerate.  We also describe how to
transform \tpt-free posets into permutations, without resorting to
ascent sequences. 

%%%%%%%%%%%%%%%%%%%%%%%%%%%%%%%%%%%%%%%%%%
\subsection{Modified ascent sequences}
%%%%%%%%%%%%%%%%%%%%%%%%%%%%%%%%%%%%%%%%%%
Let $x=(x_1,x_2,\dots,x_n)$ be any finite sequence of integers.  
We denote by $\Asc(x)$  the (ordered) list of positions where an ascent
occurs:
$$
\Asc(x) = \big(\,i : i\in[n-1] \text{ and } x_i<x_{i+1}\,\big);
$$ 
so $\asc(x)=|\!\Asc(x)|$. In terms of an algorithm we shall now
describe a function from integer sequences to integer sequences.
Let $x=(x_1,x_2,\dots,x_n)$ be the input sequence
and suppose that $\Asc(x)= (a_1,\ldots, a_k)$. Do\medskip\\
\noindent
\begin{minipage}{30em}
\mbox{}{\tt for} $i=a_1,\ldots , a_k$: \\
\mbox{}\qquad{\tt for} $j=1,\ldots ,i-1$: \\
\mbox{}\qquad\qquad{\tt if} $x_j \geq x_{i+1}$ {\tt then} $x_j := x_j+1$
\end{minipage}
\medskip\\
and denote the resulting sequence by $\widehat x$. Assuming that $x$
is an ascent sequence we call $\widehat x$ the \emph{modified ascent
  sequence}. As an example, consider the ascent sequence $x=
(0,1,0,1,3,1,1,2)$.  We have $\Asc(x)=(1,3,4,7)$ and the algorithm
computes the modified ascent sequence $\widehat x$ in the following
steps: \newcommand{\f}[1]{\makebox[2ex]{#1}}
$$
\begin{array}{ccccc}
x &\!\!\!\!=\!\!\!\!\!
& \f 0 \f{\bf 1} \f 0 \f 1 \f 3 \f 1 \f 1 \f 2\\
&& \f 0 \f 1 \f 0 \f{\bf 1} \f 3 \f 1 \f 1  \f 2\\
&& \f 0 \f 2 \f 0 \f 1 \f{\bf 3} \f 1 \f 1  \f 2\\
&& \f 0 \f 2 \f 0 \f 1 \f 3 \f 1 \f 1  \f {\bf 2}\\
&& \f 0 \f 3 \f 0 \f 1 \f 4 \f 1 \f 1 \f 2 &\!\!\!\!=\!\!\!& \widehat x
\end{array}
$$
In each step every element strictly to the left of and weakly
larger than the boldface letter is incremented by one. 
Observe that the positions of ascents in $x$ and $\widehat
x$ coincide, and  that the number of ascents in
$x$ (or $\widehat x$) is
$\asc(x)=\asc(\widehat x)=\max(\widehat x)$. 
The above procedure is easy to invert:\medskip\\
\noindent
\begin{minipage}{30em}
\mbox{}{\tt for} $i=a_k,\ldots, a_1$: \\
\mbox{}\qquad{\tt for} $j=1,\ldots ,i-1$: \\
\mbox{}\qquad\qquad{\tt if} $x_j > x_{i+1}$ {\tt then} $x_j := x_j-1$
\end{minipage}
\medskip\\
Thus the map $x\mapsto\widehat x$ is injective. 

We can also construct modified ascent sequences recursively as
follows: the only such sequence of length 1 is $(0)$. For $n\ge 2$,
$(y_1, \dots, y_n)$ is a modified ascent sequence if, and only if,
\begin{itemize}
\item $0\le y_n \le y_{n-1}$ and $(y_1, \dots, y_{n-1})$ is a
  modified ascent sequence, or
\item $ y_{n-1}<y_n\le 1 +\asc(y_1, \dots, y_{n-1})$, $y_j\not = y_n$
  for all $j<n$, and 
  $$(\,y_1-\epsilon_1,\, \dots,\, y_{n-1}-\epsilon_{n-1}\,)
  $$ 
  is a modified ascent sequence, where $\epsilon_j=1$ if $y_j\geq y_n$,
  and $\epsilon_j=0$ otherwise.
\end{itemize}

The modified ascent sequence $\widehat x$ is related to the level
distribution of the poset $P$ associated with $x$. First, observe that
the removal operation of Section~\ref{sec:rem} induces a canonical
labelling of the size $n$ poset $P$ by elements of $[n]$: the first
element that is removed gets label $n$, and so on.
Applying this to the poset of Example~\ref{ex:rem} we get the
following labelling:
\smallskip
$$
\includegraphics[scale=0.66]{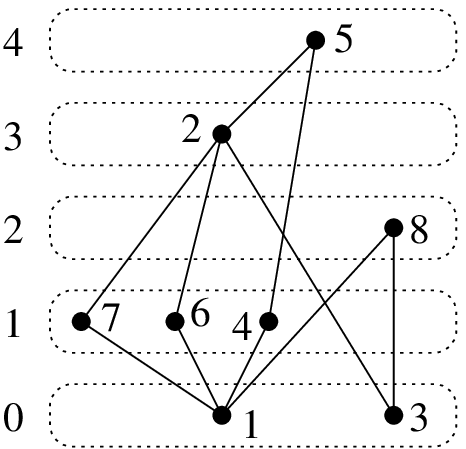}
$$ 
The following lemma is easily proved by induction, by combining the
descriptions of the map $x \mapsto \widehat x$ and of the
recursive bijection  between ascent sequences and \tpt-free posets.

\begin{lemma}\label{lem:poset-mod}
Let $P$ be a \tpt-free poset equipped with its canonical
labelling. Let  $x$ be the associated ascent sequence, and
$\widehat x =(\widehat x_1, \dots, \widehat x_n)$ the corresponding
modified ascent sequence. Then for all $i\le n$, the element $i$ of
the poset lies at level $\widehat x_i$. 
\end{lemma}
For instance,  listing the elements 
of the poset above and their respective levels gives
$$
\begin{array}{l}
\f 1 \f 2 \f 3 \f 4 \f 5 \f 6 \f 7 \f 8\\
\f 0 \f 3 \f 0 \f 1 \f 4 \f 1 \f 1 \f 2 =\ \widehat x,
\end{array}
$$ 
where we recognize the modified ascent sequence of
$(0,1,0,1,3,1,1,2)=\Psi(P)$.

%%%%%%%%%%%%%%%%%%%%%%%%%%%%%%%%%%%%%
\subsection{From posets to permutations}
%%%%%%%%%%%%%%%%%%%%%%%%%%%%%%%%%%%%%
The canonical labelling of the poset $P$ can also be used to set up
the bijection from {\tpt}-free posets to permutations of $\R$ without
using ascent sequences. We read the elements of the poset by
increasing level, and, for a fixed level, in descending order of their
labels. This gives a permutation $f(P)$.  In our example we get
$31764825$, which is the permutation of $\R_8$ associated with the
ascent sequence $(0,1,0,1,3,1,1,2)=\Psi(P)$. Let us prove that this
works in general.

\begin{proposition}\label{prop:poset-perm}
For any \tpt-free poset $P$ equipped with its canonical labelling,
the permutation $f(P)$ described above is the permutation of $\R$
corresponding to the ascent sequence $\Psi(P)$. In other words,
$$
 \Lambda^{-1} \circ \Psi(P)= \widehat{L}_0\widehat{L}_1 \dots \widehat{L}_{\rank(P)}:=\pi,
$$
where $\widehat{L}_j$ is the word obtained by reading
  the elements of $L_j(P)$ in decreasing order. Moreover, the active
  sites of the above  permutation   are those preceding and following 
$\pi$, as well as the sites separating two consecutive factors
$ \widehat{L}_j$.
\end{proposition}
\begin{proof}
  We proceed by induction on the size of $P$.  The base case $n=1$ is
  easy to check. So let $n\ge 2$, and assume the proposition holds for
  $n-1$. Let $P \in \P_{n}$ be obtained by inserting a new maximal
  element at level $i$ in $Q \in \P_{n-1}$.  By the induction
  hypothesis, the permutation corresponding to $Q$ is
  $$
  \tau = \widehat{L}'_0\widehat{L}'_1 \dots \widehat{L}'_{\rank(Q)},
  $$
  where $\widehat{L}'_j$ is obtained by reading in decreasing order the
  elements of $L_j(Q)$.
  Returning to the description of the addition operation, we see that,
  if $i \le \srank(Q)$,
  $$
  \widehat{L}_j= 
  \begin{cases}
    \widehat{L}'_j             & \text{if } j\not =i,\\
    \{n\} \cup \widehat{L}'_i  & \text{if } j=i,
  \end{cases}
  $$
  while if $i > \srank(Q)$,
  $$
  \widehat{L}_j =
  \begin{cases}
    \widehat{L}'_j     & \text{if } j< i,\\
    \{n\}              & \text{if } j= i,\\
    \widehat{L}'_{j-1} & \text{if } j> i.    
  \end{cases}
  $$
  In both cases, the word obtained by reading the elements of $P$ is
  $$
  f(P) = \widehat{L}'_0\dots \widehat{L}'_{i-1} \,n\,
  \widehat{L}'_i\widehat{L}'_{i+1}\dots \widehat{L}'_{\rank(Q)},
  $$ 
  which is obtained by inserting $n$ in the active site labeled $i$ of
  $\tau$. Hence $f(P)=\Lambda^{-1} \circ \Psi(P)$. It is then easy to
  check that the active sites of $f(P)$ are indeed those separating
  the factors $\widehat{L}_j$, and those preceding and following $f(P)$.
\end{proof}

%%%%%%%%%%%%%%%%%%%%%%%%%%%%%%%%%%%%%
 \subsection{From ascent sequences to permutations, and vice-versa}
%%%%%%%%%%%%%%%%%%%%%%%%%%%%%%%%%%%%%
By combining Lemma~\ref{lem:poset-mod} and
Proposition~\ref{prop:poset-perm}, we obtain a non-recursive
description of the bijection between ascent sequences and permutations
of $\R$.  Let $x$ be an ascent sequence, and $\widehat x$ its modified
sequence.  Take the sequence $\widehat x$ and write the numbers $1$
through $n$ below it. In our running example, $x=(0,1,0,1,3,1,1,2)$,
this gives
$$
\begin{array}{l}
\f{$\widehat x$} \makebox[1.2em]{$=$}
\f 0 \f 3 \f 0 \f 1 \f 4 \f 1 \f 1 \f 2 \\
\f{}\makebox[1.2em]{}
\f 1 \f 2 \f 3 \f 4 \f 5 \f 6 \f 7 \f 8.
\end{array}
$$ 
Let $P$ be the poset associated with $x$. By
Lemma~\ref{lem:poset-mod}, the element labeled $i$ in $P$ lies at
level $\widehat x_i$. This information is not sufficient to
reconstruct the poset $P$ but it \emm is, sufficient to reconstruct
the word $f(P)$ obtained by reading the elements of $P$ by increasing level:
Sort the pairs $\binom{\widehat x_i}{i}$ in ascending order with
respect to the top entry and break ties by sorting in descending order
with respect to the bottom entry. In the above example, this gives
$$
\begin{array}{l}
\f{}\makebox[1.2em]{}\f 0 \f 0 \f 1 \f 1 \f 1 \f 2 \f 3 \f 4 \\
\f{}\makebox[1.2em]{}\f 3 \f 1 \f 7 \f 6 \f 4 \f 8 \f 2 \f 5.   
\end{array}
$$ 
By Proposition~\ref{prop:poset-perm}, the bottom row, here $31764825$,
is the permutation $\Lambda^{-1}(x)$. We have thus established the
following direct description of $\Lambda^{-1}$.

\begin{corollary}\label{coro:ascent-perm}
Let  $x$ be an ascent sequence. Sorting the pairs
$\binom{\widehat x_i}{i}$ in the order  described above gives the
permutation $\pi=\Lambda^{-1}(x)$.
Moreover, the number of entries of $\pi$ between the active sites $i$
and $i+1$ is the number of entries of $\widehat x$ equal to $i$, for
all $i\ge 0$.
\end{corollary}

The second statement gives a non-recursive way of deriving
$x=\Lambda(\pi)$ (or, rather, $\widehat x$) from $\pi$.  Take a
permutation $\pi\in \R_n$, and indicate its actives sites. For
instance, $\pi = _0\!31_1764_28_32_45_5$. Write the letter $i$ below
all entries $\pi_j$ that lie between the active site labeled $i$ and
the active site labeled $i+1$:
$$
\begin{array}{l}
\f 3 \f 1 \f 7 \f 6 \f 4 \f 8 \f 2 \f 5\\
\f 0 \f 0 \f 1 \f 1 \f 1 \f 2 \f 3 \f 4.
\end{array}
$$
Then sort the pairs $\binom{\pi_j}i$ by increasing order of the
$\pi_j$:
$$
\begin{array}{l}
\f 1 \f 2 \f 3 \f 4 \f 5 \f 6 \f 7 \f 8\\
\f 0 \f 3 \f 0 \f 1 \f 4 \f 1 \f 1 \f 2.
\end{array}
$$
We have recovered, on the bottom row, the modified ascent sequence
$\widehat x$ corresponding to $\pi$.

%%%%%%%%%%%%%%%%%%%%%%%%%%%%%%%%%%%%%%%%%%%%%%
\subsection{Permutations avoiding \lara\ and self modified ascent
sequences}
\label{sec:lara-mod}
%%%%%%%%%%%%%%%%%%%%%%%%%%%%%%%%%%%%%%%%%%%%%%

A permutation $\pi$ avoids the barred pattern \lara\ if every occurrence of the
(classical) pattern 231 plays the role of 352 in an occurrence of the
(classical) pattern 31524. 
In other words, for every $i<j<k$ such that
$\pi_k<\pi_i<\pi_j$, there exists $\ell \in (i,j)$ and $m>k$ such that
$\pi_i \pi_\ell \pi_j\pi_k\pi_m$ is an occurrence of 31524. 
Note that every such permutation avoids the pattern $\pattern$, and
thus belongs to the set $\R$. Permutations avoiding \lara\ were
considered by  Pudwell, who gave a conjecture for their
enumeration~\cite[p.~84]{lara}. Here, we describe the ascent sequences
corresponding to  these permutations via the bijection
$\Lambda$. Then, we use this description to settle Pudwell's conjecture.

An ascent sequence $x$ is \emph{self modified} if it is fixed by the
map $x\mapsto\widehat x$ defined above. For instance,
$(0,0,1,0,2,2,0,3,1,1)$ is self modified. In view of the definition of
the map $x\mapsto \widehat x$, this means that, if $x_{i+1}>x_i$, then
$x_j<x_{i+1}$ for all $j\le i$. Recall that $\asc(x)=\max(\widehat
x)$. Combining this with the condition defining ascent sequences, we
see that $(x_1, \dots, x_n)$ is a self modified ascent sequence if and
only if $x_1=0$ and, for all $i \ge 1$, either $x_{i+1}\le x_i$ or
$x_{i+1}=1+\max\{x_j: j\le i \}$.
Consequently, a modified ascent sequence $x$ with $\max(x)=k$
reads $0A_01A_12A_2\dots\, k\, A_{k},$ where $A_i$ is a (possibly
empty) weakly decreasing factor, and each element of $A_i$ is less
than or equal to $i$.
\begin{proposition}\label{prop:lara}
  The ascent sequence $x$ is self modified if and only if the
  corresponding permutation $\pi$ avoids $3{\bar 1}52\bar{4}$. In this
  case, $\max(x)= \asc(\pi)= \rmin(\pi)-1$, where $\rmin(\pi)$ is the
  number of right-to-left minima of $\pi$, that is, the number of $i$
  such that $\pi_i<\pi_j$ for all $j>i$.
\end{proposition}
\begin{proof} 
  We proceed by induction on the size $n$ of the permutations. The
  statement is obvious for $n=1$, so let $n\ge 2$, and assume it holds
  for $n-1$. Let $\pi \in \R_n$ be obtained by inserting $n$ in the
  active site labeled $i$ of
 $\tau\in \R_{n-1}$. Let $x'=(x_1, \dots,
  x_{n-1})$ be the ascent sequence $\Lambda(\tau)$.  The ascent
  sequence $\Lambda(\pi)$ is $x=(x_1, \dots, x_{n-1},i)$.
  
  First, assume $\pi$ avoids \lara, and let us prove that $x$ is self
  modified.  Note that $\tau$ avoids \lara, because the largest entry
  in this pattern is not barred. By the induction hypothesis, the
  ascent sequence $x'=\Lambda(\tau)$ is self modified. Assume, \emm ab
  absurdo,, that $x$ is not self modified. This means that
  $x_{n-1}<i<1+\asc(x')$. That is, $n$ is inserted to the right of
  $n-1$, but not to the extreme right of $\tau$. Then the entries
  $n-1$, $n$, $\pi_n$ form an occurrence of $231$ which does not play
  the role of 352 in an occurrence of 31524 (the $\overline 4$ is
  missing). This contradicts the assumption that $\pi$ avoids
  \lara. Hence $x$ is self modified.

  Conversely, assume that $x$ is self modified (so that $x'$ itself is
  self modified), and let us prove that $\pi$ avoids \lara. By the
  induction hypothesis, $\tau$ avoids \lara.  Assume, \emm ab
  absurdo,, that $\pi$ contains an occurrence of \lara. Then this
  occurrence must contain the entry $n$, playing the role of 3 in
  231. Let $\pi_j \pi_k \pi_\ell$ be such an occurrence, with
  $n=\pi_k$. Obviously, $n$ is not inserted to the extreme right of
  $\tau$, so that $i\le x_{n-1}$. Moreover, either there is no entry
  smaller than $\pi_\ell$ between $\pi_j$ and $n$ (the entry
  $\overline 1$ is missing), or there is no entry larger than $\pi_j$
  to the right of $\pi_\ell$ (the entry $\overline 4$ is missing). In
  the first case, $\pi_{k-1}\pi_k\pi_\ell$ is another occurrence of
  \lara. Since $n$ is inserted in an active site, $\pi_{k-1}-1$ occurs
  before $\pi_{k-1}$, but then $(\pi_{k-1}-1)\pi_k\pi_\ell$ forms an
  occurrence of \lara\ in $\tau$, a contradiction. In the second case,
  $\pi_j (n-1) \pi_\ell$ forms an occurrence of \lara\ in $\tau$,
  because $n-1$ is to the right of $n$. This gives a contradiction
  again. Hence $\pi$ avoids \lara.

  Still under the assumption that $x$ is self modified, observe that
  the number of ascents, and the number of right-to-left minima,
  increase by one when going from $\tau$ to $\pi$ if
  $i=1+\asc(x')$. If $i\le x_{n-1}$, then $n$ is inserted in an ascent
  of $\tau$ (otherwise the insertion would create a forbidden
  pattern), so that the number of ascents is left unchanged. The same
  holds for the number of right-to-left minima.
\end{proof}

\begin{proposition}\label{self-mod}
  The length \gf\ of \lara-avoiding permutations is
  $$
  \sum_{k\ge 1} \frac{t^{k}}{(1-t)^{\binom{k+1}{2}}}.
  $$
  Equivalently, the number of such permutations of length $n$ is
  $$
  \sum_{k=1}^{n}\binom{\binom{k}{2}+n-1}{n-k}.
  $$ 
  Moreover, the $k$-th term of this sum counts those permutations
  that have $k$ right-to-left minima, or, equivalently, $k-1$
  ascents. This is also the number of self modified ascent sequences
  of length $n$ with largest element $k-1$.
\end{proposition}
 The corresponding numbers form Sequence
A098569 in the OEIS~\cite{sloane}.
\begin{proof} 
  By Proposition~\ref{prop:lara}, permutations of length $n$ avoiding \lara\ and
  having $k-1$ ascents are in bijection with self modified ascent
  sequences of length $n$ and largest entry $k-1$. As discussed above,
  such sequences read 
  $$
  x=0A_01A_12A_2\dots (k-1) A_{k-1},
  $$ 
  where $A_i$ is a (possibly empty) weakly decreasing factor, and
  each element of $A_i$ is less than or equal to $i$. That is,
  $$
  A_i= A_i^{(i)} A_i^{(i-1)}\dots  A_i^{(0)},
  $$
  where the factor $A_i^{(j)}$, for $j\le i$, consists of letters $j$
  only. Let $\ell_i^{(j)}$ be the length of this factor. 
  Clearly, there are $1+2+\cdots +k= \binom{k+1}2$ factors $A_i^{(j)}$
  in $x$, which may be empty. The list $(\ell_0^{(0)}, \ell_1^{(1)},
  \ell_1^{(0)}, \dots, \ell_{k-1}^{(0)})$ determines $x$ completely,
  and forms a composition of $n-k$ in $\binom{k+1}2$ (possibly empty)
  parts. Thus the number of such sequences $x$ is
  $$ \binom{n-k+ \binom{k+1}2-1}{n-k} =\binom{ \binom{k}2 +n-1}{n-k}
  $$
  as claimed.
\end{proof}

%%%%%%%%%%%%%%%%%%%%%%%%%%%%%%%%%%%%%%%%%%%%%%%%%%%%%%%%%%%%%%%%
\section{Statistics}\label{sec:stat}
%%%%%%%%%%%%%%%%%%%%%%%%%%%%%%%%%%%%%%%%%%%%%%%%%%%%%%%%%%%%%%%%

We shall now look at statistics on ascent sequences, permutations and
posets---statistics that we can translate between using our bijections.

Let $x=(x_1,x_2,\dots,x_n)$ be any sequence of nonnegative
integers. Let $\last(x)=x_n$. Define $\zeros(x)$ as the number of
zeros in $x$.  A \emph{right-to-left maximum} of $x$ 
is a letter with no larger letter to its
right; the number of right-to-left maxima is denoted $\rmax(x)$. For
example,
$$
\rmax(0,1,0,{\bf 2},{\bf 2},0,{\bf 1}) = 3;
$$ 
the right-to-left maxima are in bold.  The statistics 
\emph{right-to-left minima} ($\rmin$), \emph{left-to-right
maxima} ($\lmax$), and \emph{left-to-right minima} ($\lmin$) are
defined similarly. For sequences $x$ and $y$ of nonnegative integers,
let $x\oplus y=xy'$, where $y'$ is obtained from $y$ by adding
$1+\max(x)$ to each of its letters, and juxtaposition denotes
concatenation.  For example, $(0,2,0,1)\oplus (0,0) = (0,2,0,1,3,3)$.
We say that a sequence $x$ has $k$
\emph{components} if it is the sum of $k$, but not $k+1$, nonempty
nonnegative sequences. 
Note that  $y \oplus  z$ is  a modified ascent sequence (as defined in
Section~\ref{sec:mod}) if and only if
$y$ and $z$ are themselves modified ascent sequences. This is the case
in the above example.

For permutations $\pi$ and $\sigma$, let $\pi\oplus\sigma=\pi\sigma'$,
where $\sigma'$ is obtained from $\sigma$ by adding $|\pi|$ to each of
its letters. We say that $\pi$ has $k$ components if it is the sum of
$k$, but not $k+1$, nonempty permutations. 
Observe that $\pi\oplus\sigma$ avoids \pattern\ if and only if both
$\pi$ and $\sigma$ avoid it. This is the case for instance for
$314265= 3142\oplus 21$, which corresponds to the above modified ascent sequence
$(0,2,0,1,3,3)=(0,2,0,1)\oplus (0,0) $.

We also recall the definitions of $s(\pi)$ and $b(\pi)$. The number
of active sites of $\pi$ is $s(\pi)$. Label these
active sites  with $0$, $1$, $2$, etc.  Then $b(\pi)$ is  the
label immediately to the left of the maximal entry of $\pi$.

The number of minimal (resp.~maximal) elements of a poset $P$ is
denoted $\min(P)$ (resp. $\max(P)$). 
The ordinal sum~\cite[p. 100]{stanley} of two posets $P$ and $Q$ is
the poset $P\oplus Q$ on the union $P\cup Q$ such that $x\leq_{P\oplus
  Q} y$ if $x\leq_P y$, or $x\leq_Q y$, or $x\in P$ and $y\in Q$. 
The definition applies to labeled or unlabeled posets.
Let us say that $P$ has $k$ \emph{components} if it is the ordinal sum of
$k$, but not $k+1$, nonempty posets.
Observe that $P\oplus Q$ is  \tpt-free  if and only if both
$P$ and $Q$ are \tpt-free. 
For instance, corresponding to the modified ascent sequence
$(0,2,0,1,3,3)=(0,2,0,1)\oplus (0,0)$, above, we have
\begin{center}
  \begin{minipage}{2.6em}
    \scalebox{0.70}{\includegraphics{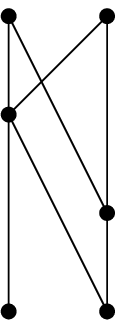}} 
  \end{minipage}
  $\quad=\quad$
  \begin{minipage}{2.6em}
    \scalebox{0.70}{\includegraphics{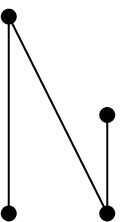}} 
  \end{minipage}
  $\quad\oplus\quad$
  \begin{minipage}{2.6em}
    \scalebox{0.70}{\includegraphics{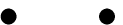}}
  \end{minipage}
\end{center}

For a \tpt-free poset $P$, a sequence $x$ and a permutation $\pi \in \R$, we
define the  following polynomials in the indeterminate $q$:
$$
\lambda(P,q) = \sum_{v\in P} q^{\rank(v)},\quad
\chi(x,q) = \sum_{i=1}^{|x|} q^{x_i},\quad
\delta(\pi,q) = \sum_{i=0}^{s(\pi)} d_iq^i,
$$ 
where $d_i$ is the number of entries of $\pi$ between the active site
labeled $i$ and the active site labeled $i+1$.
Note also that an alternative way of writing the polynomial
$\lambda(P,q)$ is $\sum_{i=0}^{\rank(P)}|L_i(P)|q^i$.
Similarly, define the polynomials 
$$
\overline \lambda(P,q) = \sum_{v\in P_{\max}} q^{\rank(v)},\quad
\overline \chi(x,q) = \sum_{x_i \rlmax} q^{x_i},\quad
\overline \delta(\pi,q) = \sum_{i=0}^{s(\pi)} \overline d_iq^i,
$$ 
where $P_{\max}$ is the set of maximal elements of $P$, the sum defining
$\overline \chi(x,q)$ is restricted to right-to-left maxima of $x$, 
and  $\overline d_i$ is the number of
right-to-left maxima of $\pi$  between the active site
labeled $i$ and the active site labeled $i+1$.

\begin{theorem}\label{thm:stats}
  Given an ascent sequence $x=(x_1,\dots,x_n)$
with modified ascent sequence $\widehat x$, 
let $P$ and $\pi$ be
  the poset and permutation corresponding to $x$ under the bijections
  described in Sections~{\rm\ref{sec:asc-av}}~and~\rm\ref{sec:poset}. Then
  $$
  \begin{matrix}
    \min(P)   &\!\!=\!\!& \zeros(x) &\!\!=\!\!& \lmin(\pi)       ;\\[.5ex]
    \srank(P) &\!\!=\!\!& \last(x)  &\!\!=\!\!& b(\pi)          ;\\[.5ex]
    \rank(P)  &\!\!=\!\!& \asc(x) &\!\!=\!\!& \asc(\pi^{-1});\\[.5ex]
    \max(P)   &\!\!=\!\!& \rmax(\widehat x) &\!\!=\!\!& \rmax(\pi);\\[.5ex]
    \comp(P)   &\!\!=\!\!& \comp(\widehat x) &\!\!=\!\!& \comp(\pi); \\[.5ex]
    \lambda(P,q) &\!\!=\!\!& \chi(\widehat x,q)&\!\!=\!\!& \delta(\pi,q);\\[.5ex]
    \overline \lambda(P,q) &\!\!=\!\!&  \overline \chi(\widehat x,q) &\!\!=\!\!&  \overline \delta(\pi,q).  
  \end{matrix}
  $$
\end{theorem}

\begin{example}
  Let $P$ be the poset from Example~\ref{ex:rem} and let $x$ and
  $\pi$ be the corresponding ascent sequence and permutation:
$$
  P=
  \begin{minipage}{7.5em}
    \includegraphics[scale=0.55]{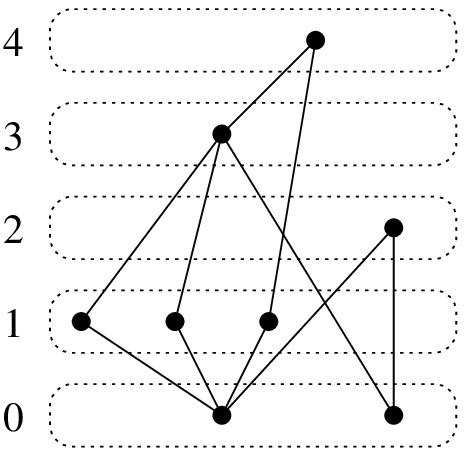}
  \end{minipage};
  \quad\;\; 
  \begin{aligned}
    x&=(0,1,0,1,3,1,1,2);\\
    \widehat x &= (0,3,0,1,4,1,1,2);
  \end{aligned}
  \quad  
  \begin{aligned}
    &\pi     = _0\!31_1764_28_32_45_5,\\
    &\pi^{-1}= 27158436.
  \end{aligned}
$$
Theorem~\ref{thm:stats} holds, with
$\min(P)      = 2$, $\srank(P)=2 $,
$\rank(P)=4 $,
$\max(P)      =2$,
$\comp(P)  =1$,
$\lambda(P,q)=q^{4} + q^{3} +  q^{2} + 3q + 2$,
and $\overline \lambda(P,q)  =q^{4} + q^{2}$.
\end{example}

\begin{proof}[Proof of Theorem~\rm\ref{thm:stats}]
  The polynomial identity $\lambda(P,q) = \chi(\widehat x,q)
  = \delta(\pi,q)$ is a consequence of Lemma~\ref{lem:poset-mod} for
  the first part, and of Corollary~\ref{coro:ascent-perm} for the
  second part.  Setting $q=0$ in $\lambda(P,q) = \chi(\widehat x,q)$
  gives $\min(P) = \zeros(x)$ (note that $\zeros(x)=\zeros(\widehat
  x)$). Setting $q=0$ in the identity $\chi(\widehat x,q)
  = \delta(\pi,q)$ shows that $\zeros(x)$ is the number of entries of
  $\pi$ between the first two active sites. Let us prove that these
  are the entries $\pi_1,\pi_2,\dots,\pi_k$, where $k$ is the largest
  integer such that $\pi_1>\pi_2>\dots>\pi_k$. Note that this means
  that $\lmin(\pi)=k$. For $1\le i <k$, the entry $\pi_i$ is followed
  by an inactive site, because $\pi_i -1$ appears to the right of
  $\pi_i$.  Assume $\pi_k>1$. Then $\pi_k-1$ appears to the right of
  $\pi_k$, but $\pi_{k+1}>\pi_k$, so that $\pi_k \pi_{k+1} (\pi_k-1)$
  is an occurrence of the forbidden pattern, a contradiction. So
  $\pi_k=1$, the site following $\pi_k$ is active, and the result is
  proved.

  The result dealing with $\last(x)$ has already been proved, when we
  established that $\Lambda$ and $\Psi$ were indeed
  bijections. See~\eqref{properties} and~\eqref{properties-bis}.
  The same holds for the connection between $\asc(x)$ and $\rank(P)$
  (see~\eqref{properties-bis} again). We also know that
  $\asc(x)=s(\pi)-2$, but we wish to relate this number to $\asc(\pi^{-1})$.
  
  The next identities will be proved by induction on $n$. These are easy
  to check when $n=1$, so we take $n\ge 2$. Denote $i=x_n$,
  $(Q,i)=\psi(P)$, and let $\tau$ be obtained by deleting the entry
  $n$ from $\pi$. Let $x'=(x_1, \dots, x_{n-1})= \Lambda(\tau)
  = \Psi(Q)$.
  
  Let us start with the connection between $\asc(x)$ and
  $\asc(\pi^{-1})$. The number of ascents increases (by one) when going from
  $\tau^{-1}$ to $\pi^{-1}$ if and only if $n$ is inserted, in $\tau$, 
to  the right of $n-1$: 
  as shown in the proof of Theorem~\ref{th:lambda}, this means that
  $\asc(x)=1+\asc(x')$ (Case 2 of the proof).

  The identity that involves $\max(P)$ is just the  case $q=1$ of the
  identity that involves the polynomial $\overline \lambda(P,q)$, which
  we now prove.
  Let us now study how the polynomials $\overline \lambda(\cdot,q)$, $\overline \chi(\widehat
  \cdot, q)$ and $\overline \delta(\cdot,q)$ evolve as the size of the
  poset/sequence/permutation increases. For posets,
  \begin{align*}
  \overline \lambda(P,q) &= 
  \begin{cases}
    \;\overline \lambda(Q,q) + q^{i} 
    & \text{if }i\le \srank(Q),\\[1ex]
    \;\displaystyle{ q^{i} + \sum_{j=i}^{\rank(Q)}\!|\overline L_j(Q)|q^{j+1}}
    & \text{if } i> \srank(Q), 
  \end{cases}\\
  \intertext{where $\overline L_j(Q)$ is the set of maximal elements
    of $Q$ at level $j$.  Similar relations hold for $\overline
    \chi(\widehat x,q)$ and $\overline \delta(\pi,q)$. Denoting
the modified ascent sequence of $x'$ by 
    $\widehat{x}'=(\widehat x'_1, \dots,\widehat x'_{n-1})$, we have
  }
  \overline \chi(\widehat x,q) &=
  \begin{cases}
    \;\overline \chi(\widehat x',q) + q^{i} 
    & \text{if }i\le x_{n-1},\\[1ex]
    \;\displaystyle{
      q^{i} + \sum_{\rlmax \widehat{x}'_{\!j}\ge i}q^{\widehat x'_{\!j}+1}}
    & \text{if } i> x_{n-1},    
  \end{cases}\\
  \overline \delta(\pi,q) &= 
  \begin{cases}
    \;\overline \delta(\tau,q) + q^{i} 
    & \text{if }i\le b(\tau),\\[1ex]
    \;\displaystyle{
      q^{i} + \sum_{j\ge i}\overline d_j\,q^{j+1}}
    & \text{if } i> b(\tau),    
  \end{cases}
  \end{align*}
  and the statement $\overline \lambda(P,q)=\overline \chi(\widehat
  x,q)=\overline \delta(\pi,q)$ follows by induction.

  We shall finally prove that $\comp(P)=\comp(\widehat
  x)=\comp(\pi)$. First, observe that it suffices to prove that
  \begin{multline*}  
    \widehat x= \widehat y\oplus \widehat z \text{ with }
    |y|=\ell  \text{ and } |z|=m \\
    \,\Leftrightarrow\,
    \pi=\sigma \oplus \tau \text{ with }
    |\sigma|=\ell  \text{ and } |\tau|=m \\
    \,\Leftrightarrow\, 
    P=P_y\oplus P_z \text{ with }
    |P_y|=\ell  \text{ and } |P_z|=m,
  \end{multline*}
  and that 
  $\sigma$ and $P_y$ (resp.~$\tau$ and $P_z$)
  are respectively the permutation  and the poset associated with the ascent
  sequence $y$ (resp.~$z$). It then follows by induction on the number
  of components, not only that $\widehat x$, $\pi$ and $P$ have the
  same number of components, but also that the sizes of the components
  are the same.
  
  From Corollary~\ref{coro:ascent-perm}, it is easily seen that
  $\pi=\sigma \oplus \tau$ if and only if $\widehat x= \widehat
  y\oplus \widehat z$, with $\Lambda(\sigma)= y$ and $\Lambda(\tau)=
  z$. Assume this holds. Let us write $\ell=|y|$ and $m=|z|$. Let $P_y$
  and $P_z$ be the posets corresponding to $y$ and $z$,
  respectively. Let us prove that the \emm canonically labeled,
  versions of $P, P_y$ and $P_z$ satisfy $P=P_y\oplus P_z$. Clearly,
  the $\ell$ first steps of the recursive construction of $P$
  (starting from the ascent sequence $x$) give the (labeled) poset
  $P_y$, which satisfies $\rank(P_y)=\max(\widehat y)$ by
  Lemma~\ref{lem:poset-mod}. Then comes the letter $x_{\ell+1}$. As 
  $\widehat x_{\ell+1}=1+\max\{\widehat x_j : j\le\ell\}$, the element
  $\ell+1$ ends up, in the final poset $P$, at a higher level than the
  elements $1, 2, \dots, \ell$.  This implies that the element
  $\ell+1$ is added using the operation \addii, and hence covers all
  maximal elements of $P_y$. Consequently, the set of predecessors of
  $\ell+1$ is $P_y$, and the poset obtained at this stage is
  $P_y\oplus \{\ell+1\}$. One then proceeds by induction of the size
  of $z$. We do not give the details.  One checks inductively that the
  relative order of the elements labeled $\ell+1$ to $n=\ell+m$ in $P$
  coincides with their order in $P_z$, and that every element of $P_y$
  is smaller than every element of $P_z$.

  Conversely, assume $P=P_y\oplus P_z$, with $\ell=|P_y|$, $m=|P_z|$
  and $\ell+m=n$. We will prove that in the canonical labelling of
  $P$, the largest $m$ letters are those of $P_z$. Again, this follows
  from an induction on $m$. As usual, we write $(Q,i)=\remove(P)$. If
  $m=1$, then $n$ is the unique maximal element of $P$, and
  $Q=P_y$. Otherwise, the element $n$ is in $P_z$ (as $P_y$ contains
  no maximal element), and one has to check that $Q=P_y\oplus P'_z$
  where $P'_z$ is obtained by applying the removal procedure to
  $P_z$. We do not give all the details. The key point is that, when
  \subiii\ is used, the set $\N=D_{i+1}\setminus D_i$ of elements that
  become maximal in $Q$ does not contain any element of $P_y$. Indeed,
  every element of $P_y$ is smaller than every element of $P_z$, so
  that it belongs to $D_i$. Once it is proved that the $m$ largest
  elements of $P$ are those of $P_z$, one applies
  Proposition~\ref{prop:poset-perm} to conclude that the corresponding
  permutation $\pi$ reads $\sigma\oplus \tau$, where $\sigma$
  (resp. $\tau$) corresponds to $P_y$ (resp.~$P_z$).
\end{proof}

%%%%%%%%%%%%%%%%%%%%%%%%%%%%%%%%%%%%%%%%%%%%%%%%%%%%%%
\section{The number of \tpt-free posets}\label{sec:gf}
%
%%%%%%%%%%%%%%%%%%%%%%%%%%%%%%%%%%%%%%%%%%%%%%%%%%%%%%
The aim of this section is to obtain a closed form expression for
the \gf\ $P(t)$ of unlabeled \tpt-free posets:
\begin{align*}
  P(t) 
  &= \sum_{n\ge 0} p_n \, t^n\\
  &= 1+ t+ 2t^2+ 5t^3+ 15t^4+ 53t^5+ 217t^6+ 1014t^7+ 5335t^8+
  O(t^{9}),
\end{align*}
where $p_n$ is the number of \tpt-free posets of cardinality $n$.
The sequence $(p_n)_{n\ge 0}$ is Sequence A022493 in the OEIS~\cite{sloane}.
\begin{theorem}\label{thm:u=1}
  The  \gf\ of unlabeled \tpt-free posets is
  $$
  P(t)
   =\sum_{n\ge 0} \ \prod_{i=1}^n \left( 1-(1-t)^i\right).
  $$
\end{theorem}
Of course, the series $P(t)$ also counts permutations of $\R$, or
ascent sequences, by length. To our knowledge, this result is
new.  El-Zahar~\cite{ZAHAR} and Khamis~\cite{smkhamis} used a
recursive description of \tpt-free posets, different from that of
Section~\ref{sec:poset}, to derive a pair of functional equations that
define the series $P(t)$. However, they did not solve these equations.
Haxell, McDonald and Thomasson~\cite{haxell} provided an
algorithm, based on a complicated recurrence relation, to produce the
first numbers $p_n$. 
However, the above series has already appeared in the literature:
it was proved by Zagier~\cite{zagier} to
 count certain involutions introduced by
Stoimenow~\cite{stoim}. (The connection between these involutions and
\tpt-free posets is the topic of the next section.)
Moreover, Zagier derived a number of
interesting properties of the series $P(t)$. In particular, he 
gave the following asymptotic estimate:
$$
\frac{p_n}{n!} \sim \kappa\,  \left(\frac 6 {\pi^2}\right)^{\!n}\!\sqrt n, 
\quad \text{where} \quad \kappa = \frac{12\sqrt 3}{\pi^{5/  2}} e^{\pi^2/12}.
$$ 
Note that since the growth constant $ 6/ {\pi^2}$ is transcendental it
follows that the \gf\ is not D-finite~\cite{stanley-vol2,wimp-zeil}.
Zagier also proved that the series $P(t)$ satisfies the following
remarkable formula:
$$
P(1-e^{-24x})= e^x \sum_{n\ge 0} \frac {T_n} {n!} x^n,
$$
where 
$$
\sum_{n\ge 0} \frac{T_n}{(2n+1)! }  x^{2n+1}= \frac{\sin 2x}{2\cos
  3x}.
$$
Our proof of Theorem~\ref{thm:u=1} exploits the   recursive
structure of ascent sequences. 
This structure  translates into a functional equation
for the \gf\ of these sequences, which is solved by the so-called kernel
method. This gives a closed form expression of a
bivariate \gf, which counts ascent sequences by their length and 
ascent number. However, one still needs to transform this expression
to obtain the above expression for the length \gf.

%%%%%%%%%%%%%%%%%%%%%%%%%%%%%%%%%%%%%%%%%%%%%%%%
\subsection{The functional equation}
%%%%%%%%%%%%%%%%%%%%%%%%%%%%%%%%%%%%%%%%%%%%%%%%

Let $F(t;u,v)\equiv F(u,v)$ be the \gf\ of ascent sequences, counted by length
(variable $t$), number of ascents (variable $u$) and 
last entry (variable $v$). This is a \fps\ in $t$ with
coefficients in $\QQ[u,v]$. The first few terms of $F(t;u,v)$ are
$$
F(t;u,v)=1+t+(1+uv)t^2+ (1+2uv+u+u^2v^2)t^3+ O(t^4).
$$
Let $G(t;u,v)=F(t;u,v)-1\equiv G(u,v)$ count non-empty ascent sequences. We write
$$
G(t;u,v)=  \sum_{a,\ell \ge 0} G_{a,\ell}(t) u^a v^\ell,
$$
so that $G_{a,\ell}(t)$ is the length \gf\ of sequences having $a$
ascents and ending with the value $\ell$.

\begin{lemma}\label{lem:eq-func}
  The \gf\ $G(t;u,v)$ satisfies
$$
\left(v-1-tv(1-u)\right) G(u,v)=t(v-1)-tG(u,1)+tuv^2G(uv,1).
$$
Equivalently, $F(t;u,v)=1+G(t;u,v)$ satisfies
$$
\left(v-1-tv(1-u)\right) F(u,v)=(v-1)(1-tuv)-tF(u,1)+tuv^2F(uv,1).
$$
\end{lemma}
\begin{proof}
  Let $x'=(x_1, \dots, x_{n-1})$ be a non-empty ascent sequence with
  $a$ ascents, ending with the value $x_{n-1}=\ell$. Then $x=(x_1,
  \dots, x_{n-1}, i)$ is an ascent sequence if and only if $i\in [0,
    a+1]$.  Moreover, the sequence $x$ has $a$ ascents if $i\le\ell$,
  and $a+1$ ascents otherwise. Given that $(0)$ is the only ascent
  sequence of length $1$, this gives:
  \begin{align*}
    G(u,v) &= t+ t\sum_{a,\ell\ge 0} G_{a,\ell}(t)\left( \sum_{i=0}^\ell u^a v^i +
    \sum_{i=\ell+1}^{a+1} u^{a+1} v^i \right)\\
    &= t+ t\sum_{a,\ell\ge 0} G_{a,\ell}(t)u^a\left( \frac {v^{\ell+1}-1}{v-1}+u\, 
    \frac{v^{a+2}-v^{\ell+1}}{v-1} \right)\\
    &= t + t\, \frac{vG(u,v)-G(u,1)}{v-1} +{tuv}\frac{vG(uv,1)-G(u,v)}{v-1} .
  \end{align*}
  The result follows.
\end{proof}

\noindent
{\bf Remark.} The variables $u$ and $v$ are needed to transform  our
recursive description of ascent sequences into a functional equation,
and are thus \emm catalytic,, in the sense of~\cite{zeil-umbral}. Setting $v=1$
 in the equation gives a tautology. Setting $u=1$ gives a relation
 between $G(1,v)$, $G(1,1)$ and $G(v,1)$ which does not suffice to
 characterize these series.

%%%%%%%%%%%%%%%%%%%%%%%%%%%%%%%%%%%%%%%%%%%%%%%%
\subsection{The kernel method}
%%%%%%%%%%%%%%%%%%%%%%%%%%%%%%%%%%%%%%%%%%%%%%%%
Consider the functional equation satisfied by $F(t;u,v)$ given by
Lemma~\ref{lem:eq-func}. The coefficient of $F(u,v)$, called the \emm
kernel,, vanishes when $v=V(u)$, with $V(u)= 1/(1-t+tu)$.
Recall that $F(t;u,v)$ is a series in $t$ with coefficients in
$\QQ[u,v]$. Hence $F(u,V(u))$ is a well-defined series in $t$ with
coefficients if $\QQ[u]$. Replacing $v$ by $V(u)$ in the functional
equation cancels the left-hand side, and results in:
$$
F(u,1) = \frac{(1-u)(1-t)}{(1-t+tu)^2}+ \frac{u}{(1-t+tu)^2}
F\!\left(\frac u{1-t+tu},1\right).
$$
Iterating this equation gives
\begin{align*}
  F(u,1) 
  &\,=\, \frac{(1-u)(1-t)}{(1-t+tu)^2}
  \,+\,\frac{u(1-t)^2(1-u)}{(1-t+tu)(1-2t+2tu+t^2-t^2u)^2}\\
  &\quad\,+\, \frac{u^2}{(1-t+tu)(1-2t+2tu+t^2-t^2u)^2}
  F\!\left(\frac u{1-2t+2tu+t^2-t^2u},1\right)\\
  &\,=\, \sum_{k=1}^n \frac{(1-u) u^{k-1}(1-t)^k}{(u-(u-1)(1-t)^k)
    \prod_{i=1}^k(u-(u-1)(1-t)^i) }\\
  &\quad\,+\,\frac{u^n}{(u-(u-1)(1-t)^n)\prod_{i=1}^n(u-(u-1)(1-t)^i)}
  F\!\left(\frac u {u-(u-1)(1-t)^n},1\right).
\end{align*}
Letting $n\rightarrow \infty$, we obtain a first expression of
$F(t;u,1)$, as a formal series in $u$ with rational coefficients in
$t$.
\begin{proposition}\label{prop:series-u}
  The series $F(t;u,1)$ counting ascent sequences by their length and
  ascent number, seen as a series in $u$, has rational
  coefficients in $t$, and satisfies
  $$
  F(t;u,1) = \sum_{k\ge 1}\frac{(1-u)\, u^{k-1}(1-t)^k}{(u-(u-1)(1-t)^k)
    \prod_{i=1}^k(u-(u-1)(1-t)^i) }.
  $$
\end{proposition}
Alas, the above expression is only
convergent \emm as a series in, $u$. In particular, if we set $u=1$,
the result seems to be zero (because of the factor $(1-u)$). If we
ignore this factor, what remains reads
$$
\sum_{k\ge 1}(1-t)^k,
$$
which is not a convergent series in the formal variable $t$.
We will now work out another series expression of $F(t;u,1)$,
which converges as a series in $t$ with coefficients in $\QQ[u]$. In
this expression we can set $u=1$, and this will give Theorem~\ref{thm:u=1}.

%%%%%%%%%%%%%%%%%%%%%%%%%%%%%%%%%%%%%%%%%%%%%%%%
\subsection{Transforming the solution}
%%%%%%%%%%%%%%%%%%%%%%%%%%%%%%%%%%%%%%%%%%%%%%%%
Our first lemma tells us that certain series, which look
like the one in Proposition~\ref{prop:series-u}, are actually
polynomials in $u$ and $t$.
\begin{lemma}\label{lem:polynom}
  Let $m\ge 1$ be an integer. Let $S(t;u)$ be the following series in
  $u$, with rational coefficients in $t$:
  $$
  S(t;u)= \sum_{k\ge 1} \frac{(u-1)^m\, u^{k-1}(1-t)^{mk}}{
    \prod_{i=1}^k(u-(u-1)(1-t)^i) }.
  $$
  Then $S(t;u)$ is actually a polynomial in $u$ and $t$:
  $$
  S(t;u)= - \sum_{j=0} ^{m-1} (u-1)^j u^{m-1-j} (1-t)^j
  \prod_{i=j+1}^{m-1} \left( 1-(1-t)^i\right).
  $$
\end{lemma}
\begin{proof}
  Consider the following equation in $\Phi(t;u)\equiv \Phi(u)$:
  $$
  \Phi(u)= \frac{(u-1)^m(1-t)^m}{1-t+tu} + u(1-t+tu)^{m-2}\,
  \Phi\!\left(\frac u {1-t+tu}\right).
  $$
  By iterating it, we see that it has a unique solution in the space of
  series in $u$ with rational coefficients in $t$, and that this
  solution is the first expression of $S(t;u)$ given above.
  Moreover, by writing the equation as follows:
  $$
  (1-t+tu) \Phi(u)= {(u-1)^m(1-t)^m} + u(1-t+tu)^{m-1}\,
  \Phi\!\left(\frac u {1-t+tu}\right),
  $$
  one checks easily that the second expression of $S(t;u)$ (a
  polynomial in $t$ and $u$) is also a solution. Since a polynomial
  in $t$ and $u$ is (also) a series in $u$ with rational coefficients
  in $t$, the identity is established.
\end{proof}

From the above lemma, we are going to derive another expression of the
series $F(t;u,1)$, in which the substitution $u=1$ raises no
difficulty.
\begin{theorem}\label{thm:u}
  Let $n\ge 0$, and consider the following polynomial in $t$ and $u$:
  $$
  F_n(t;u)= \sum_{\ell=0}^n (u-1)^{n-\ell} u^\ell \sum_{m=\ell} ^{n}
  (-1)^{n-m} {n\choose {m}} (1-t)^{m-\ell }
  \prod_{i=m-\ell+1}^{m} \left( 1-(1-t)^i\right).
  $$
  Then $F_n(t;u)$ is a multiple of $t^n$. Moreover,
  the \gf\ of ascent sequences, counted by the length and the ascent
  number, is
  $$
  F(t;u,1)= \sum_{n\ge 0} F_n(t;u).
  $$  
  When $u=1$, 
  $$
  F_n(t;1)=\prod_{i=1}^{n} \left( 1-(1-t)^i\right),
  $$
  and Theorem~\ref{thm:u=1} follows.
\end{theorem}
\begin{proof}
  We return to the expression of $F(t;u,1)$ given in
  Proposition~\ref{prop:series-u}. The expansion
  $$
  \frac 1 {u-(u-1)(1-t^k)} = \frac 1 {1-(u-1)((1-t)^k-1)}
  =\sum_{n\ge 0} (u-1)^n ((1-t)^k-1)^n
  $$
  is valid in the space of series in $t$ with polynomial coefficients in
  $u$, as $(1-t)^k-1=O(t)$. It holds as well in the larger space of
  \fps\ in $t$ and $u$. Moreover, the $n$th term is $O(t^n)$. Hence, in
  the space of series in $t$ and $u$,
  $$
  F(t;u,1)=  \sum_{k\ge 1} \frac{(1-u)\, u^{k-1}(1-t)^k}{
    \prod_{i=1}^k(u-(u-1)(1-t)^i) }
  \sum_{n\ge 0} (u-1)^n ((1-t)^k-1)^n
  =\sum_{n\ge 0} F_n(t;u)
  $$
  where 
  \begin{align*}
    F_n(t;u)&= -\sum_{k\ge 1} \frac{(u-1)^{n+1}\, u^{k-1}(1-t)^k}{
      \prod_{i=1}^k(u-(u-1)(1-t)^i) }((1-t)^k-1)^n\\
    &=-\sum_{k\ge 1} \frac{(u-1)^{n+1}\, u^{k-1}(1-t)^k}{
      \prod_{i=1}^k(u-(u-1)(1-t)^i) }\sum_{m=0}^n {n\choose
      m}(1-t)^{km}(-1)^{n-m}\\
    &=-\sum_{m=0}^n {n\choose m}(-1)^{n-m}(u-1)^{n-m}\sum_{k\ge 1}
    \frac{ (u-1)^{m+1}\, u^{k-1}
      (1-t)^{k(m+1)}}{  \prod_{i=1}^k(u-(u-1)(1-t)^i) }.
  \end{align*}
  It remains to apply Lemma~\ref{lem:polynom}, with $m$ replaced by $m+1$:
  $$
  F_n(t;u) = \sum_{m=0}^n (-1)^{n-m}{n\choose m}(u-1)^{n-m}
  \sum_{j=0} ^{m} (u-1)^j u^{m-j} (1-t)^j
  \prod_{i=j+1}^{m} \left( 1-(1-t)^i\right).
  $$
  The expected expression of $F_n(t;u)$ follows, upon writing $j=m-\ell$.
\end{proof}

%%%%%%%%%%%%%%%%%%%%%%%%%%%%%%%%%%%%%%%%%%%%%%%%%%%%%%%%%%
\section{Involutions with no neighbour nesting}
\label{sec:chord}
%%%%%%%%%%%%%%%%%%%%%%%%%%%%%%%%%%%%%%%%%%%%%%%%%%%%%%%%%%

As discussed above, the series of Theorem~\ref{thm:u=1} is known to
count certain involutions on
$2n$ points, called \emm regular linearized chord diagrams,\ (RLCD) by
Stoimenow~\cite{stoim}. This result was proved by
Zagier~\cite{zagier}, following Stoimenow's paper. 
In this section, we give a new proof of Zagier's result, by
constructing a bijection between  RLCDs on 
$2n$ points  and unlabeled $({\mathbf{2+2}})$-free posets
of size $n$. 

Let $\Invs{2n}$ be the collection of involutions $\pi$ in $\sym_{2n}$
that have no fixed points and for which every descent
crosses the main diagonal in its dot diagram.
Equivalently, if  $\pi_i> \pi_{i+1}$ then $\pi_i>i\ge\pi_{i+1}$.
 An alternative description  can be given in terms of the
\emm chord diagram, of $\pi$, which is obtained by joining the points
$i$ and $\pi_i$ by a chord (Figure~\ref{fig:chords}, top left).
Indeed,  $\pi \in \Invs{2n}$ if and only if, for any $i$, the chords
attached to $i$ and $i+1$ are not \emm nested,, in the terminology used
recently for partitions and involutions (or matchings)~\cite{chen,krattenthaler}. That is, the  configurations
shown on the left of the rules of Figure~\ref{fig:swapping} are forbidden (but a
 chord linking $i$ to $i+1$ is allowed). Such involutions were called
\emm regular linearized chord diagrams, by Stoimenow. We prefer to say that they have no \emm neighbour nesting.,

Recall that a poset $P$ is \tpt-free if and only if it is an \emm
interval order,~\cite{FISH_OPER}. This means that there exists a
collection of intervals on the real line whose relative order is $P$,
under the 
 relation:
\beq\label{int-order}
[a_1,a_2] < [a_3,a_4] \;\,\Longleftrightarrow\,\; a_2<a_3.
\eeq
Let 
$\pi$ be a fixed point free involution with transpositions
$\{(\alpha_i,\beta_i)\}_{i=1}^n$ where 
$\alpha_i<\beta_i$ for all $i$.
Define $\Omega(\pi)$ to be the interval order (or equivalently,
\tpt-free poset) 
associated with the
collection of intervals $\{[\alpha_i,\beta_i]\}_{i=1}^n$. The
transformation $\Omega$ has a  symmetry property that will
be important: the poset associated with the mirror of
  $\pi$ (obtained by reflecting the chord diagram of $\pi$ across
a vertical line) is the dual of $\Omega(\pi)$.

  \begin{figure}[h]
    \centerline{\scalebox{0.8}{\includegraphics{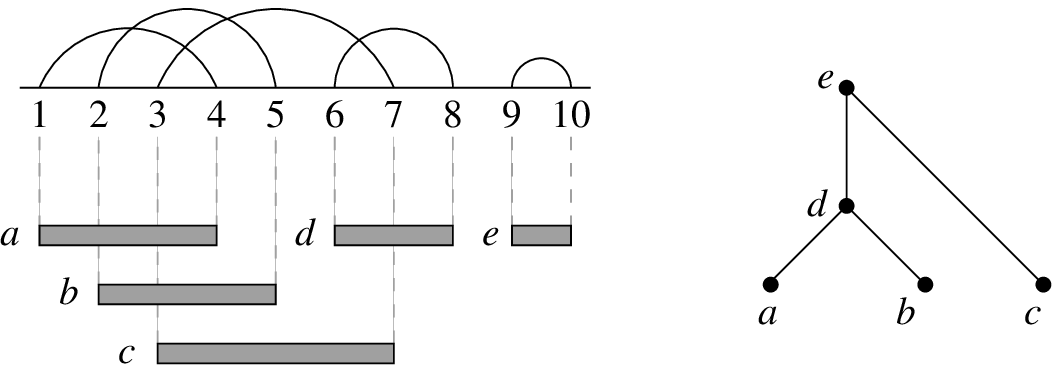}}}
    \caption{The involution  $\pi = 4 \, 5\, 7 \, 1 \, 2 \, 8 \, 3 \, 6 \, 10 \, 9
   \in \Invs{10}$, 
	the corresponding collection of intervals
      and the associated \tpt-free poset.}
    \label{fig:chords}
  \end{figure}

\begin{example}
  Consider $\pi \;=\; 4 \, 5\, 7 \, 1 \, 2 \, 8 \, 3 \, 6 \, 10 \, 9
  \; \in \; \Invs{10}$.  The transpositions of $\pi$ are shown in the
  chord diagram of Figure~\ref{fig:chords}.  Beneath the chord diagram
  is the collection of intervals that corresponds to $\pi$,
  and the $(\mathbf{2+2})$-free poset $\Omega(\pi)$ is
  shown on the right-hand side. 
We have added labels to highlight the correspondence between
intervals and poset elements.
\end{example}

\begin{theorem}\label{thm:inv-poset}
  The map $\Omega$, restricted to involutions with no neighbour
  nesting,  induces   a bijection between involutions of $\Invs{2n}$ and
  \tpt-free posets on $n$ elements. 
\end{theorem}
\begin{proof}
 Let us first prove that the restriction of $\Omega$ is a surjection. That
is, for every poset  $P\in \P_n$, one can find an involution
$\pi\in \Invs{2n}$ such
that $\Omega(\pi)=P$. Let $P \in \P_n$. As $P$ is an interval order,
there exists a collection of $n$ intervals  on the real line whose
relative order is $P$, under the order relation~\eqref{int-order}. We
can assume that the (right and left) endpoints of these  $n$ intervals
are $2n$ distinct points. Indeed, if 
some point $x$ occurs $k$ times as an endpoint, 
then the intervals ending at $x$ are incomparable,
and one can replace $x$ by $k$ distinct points and obtain a new collection
of intervals  whose order is still $P$, as shown below. 

\vskip 4mm
    \centerline{\scalebox{0.8}{\includegraphics{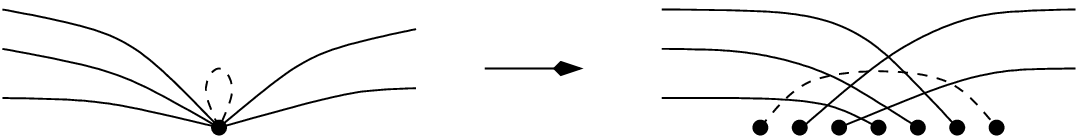}}}

Note that in this figure, intervals are
represented by chords rather than segments for the sake of clarity.  In
particular, an interval reduced to one point is represented by a loop.

 \begin{figure}[h!]
 \scalebox{0.7} {\input{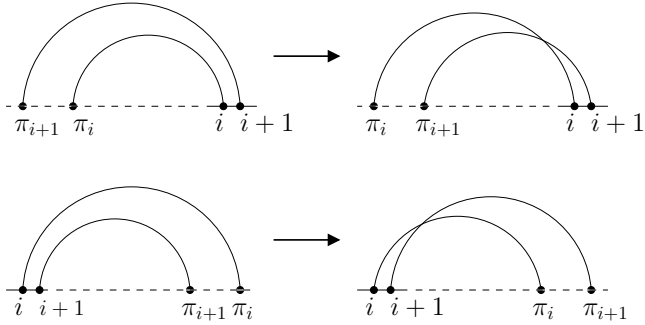}}
    \caption{Two operations  on chord diagrams.}
    \label{fig:swapping}
  \end{figure}

Clearly, we can then assume that the $2n$ distinct endpoints
of our $n$ intervals are exactly $1,2, \ldots, 2n$. These intervals thus
form a chord diagram, and there exists a fixed point free involution
$\pi\in \sym_{2n}$ such that $\Omega(\pi)=P$. However, $\pi$ may have
neighbour nestings. Transform recursively every such nesting as shown
in Figure~\ref{fig:swapping}. The corresponding poset does not change
with these transformations, while the number of crossings in the chord
diagram increases. Hence the sequence of transformations must stop,
and when it stops we have obtained an involution $\pi'$ with no
neighbour nesting such that $\Omega(\pi')=P$. An example is shown in
Figure~\ref{fig:nestings}, where we have indicated by a white dot
which nesting is transformed.

 \begin{figure}[h!]
\centerline{\scalebox{0.8}{\includegraphics{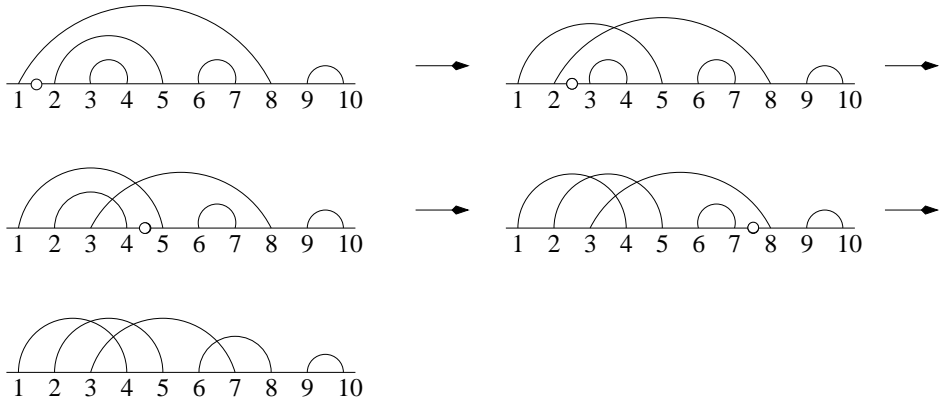}}}
    \caption{Deleting neighbour nestings from an involution.}
    \label{fig:nestings}
  \end{figure}

Let us now prove that $\Omega$, restricted to $\Invs{2n}$, is
injective. Assume $\pi\in \Invs{2n}$ and $\Omega(\pi)=P$. We will
prove that one can reconstruct the chord diagram of $\pi$ from $P$.

 We associate with $\pi$ a word $u=u_1 u_2\cdots u_{2n}$ over the
alphabet $\{o,c\}$ as follows: $u_i=o$ (resp. $c$) if there is an
opening (resp. closing) chord at $i$. That is, if $\pi_i>i$ (resp.
$\pi_i<i$).  We define an \emm opening run, to be a maximal factor of
$u$ containing only the letter $o$. We define similarly closing
runs. For instance, the involution in Figure~\ref{fig:chords} has 3
opening runs (and consequently 3 closing runs).

The order $P=\Omega(\pi)$ can be seen as an order on the chords of
$\pi$: given a chord $a=(i,j)$, with $i<j$, the chords that are
smaller than $a$ (the \emm predecessors, of $a$) are those that close
before $i$, and the chords that are larger than $a$ are those that
open after $j$. From this observation, it follows easily, by induction
on $i$, that the \emm level, of $a$ in $P$ (as defined in
Section~\ref{sec:poset}) is the number of closing runs found before
$i$ in $u$. Let $k=\rank(P)$ be the highest level of an element of
$P$, and for $0\le i \le k$, denote by $m_i$
 the number of elements at level $i$ in $P$. 
Then  the preceding discussion implies that the word $u$
associated with $ \pi$ is of the form $o^{m_0} c^{n_k}o^{m_1} c^{n_{k-1}}\cdots 
o^{m_k} c^{n_0}$ where $n_i>0$ for all $i$.
But by symmetry, $n_i$ must be the number of elements at level $i$ in
$P^*$ (and moreover, $\rank(P)=\rank(P^*)$). Thus the word $u$ can be
reconstructed from $P$ and its dual. We represent $u$ by a
sequence of $2n$ half-chords, some opening, some closing. For instance,
we show below
the sequence of half-chords obtained from the
poset $P$ of Figure~\ref{fig:chords} and its dual $P^*$. It is
convenient to assign with each opening (resp. closing) half-chord a
label, equal to the level of the corresponding element of $P$ (resp. $P^*$).

\smallskip
    {\begin{picture}(0,0)%
\includegraphics{half-chords.pstex}%
\end{picture}%
\setlength{\unitlength}{4144sp}%
\begingroup\makeatletter\ifx\SetFigFontNFSS\undefined%
\gdef\SetFigFontNFSS#1#2#3#4#5{%
  \reset@font\fontsize{#1}{#2pt}%
  \fontfamily{#3}\fontseries{#4}\fontshape{#5}%
  \selectfont}%
\fi\endgroup%
\begin{picture}(4838,1968)(2866,-3790)
\put(5806,-3616){\makebox(0,0)[lb]{\smash{{\SetFigFontNFSS{10}{12.0}{\familydefault}{\mddefault}{\updefault}{\color[rgb]{0,0,0}1}%
}}}}
\put(4006,-3616){\makebox(0,0)[lb]{\smash{{\SetFigFontNFSS{10}{12.0}{\familydefault}{\mddefault}{\updefault}{\color[rgb]{0,0,0}0}%
}}}}
\put(4366,-3616){\makebox(0,0)[lb]{\smash{{\SetFigFontNFSS{10}{12.0}{\familydefault}{\mddefault}{\updefault}{\color[rgb]{0,0,0}0}%
}}}}
\put(4726,-3616){\makebox(0,0)[lb]{\smash{{\SetFigFontNFSS{10}{12.0}{\familydefault}{\mddefault}{\updefault}{\color[rgb]{0,0,0}0}%
}}}}
\put(7246,-3616){\makebox(0,0)[lb]{\smash{{\SetFigFontNFSS{10}{12.0}{\familydefault}{\mddefault}{\updefault}{\color[rgb]{0,0,0}0}%
}}}}
\put(6526,-3616){\makebox(0,0)[lb]{\smash{{\SetFigFontNFSS{10}{12.0}{\familydefault}{\mddefault}{\updefault}{\color[rgb]{0,0,0}1}%
}}}}
\put(5086,-3616){\makebox(0,0)[lb]{\smash{{\SetFigFontNFSS{10}{12.0}{\familydefault}{\mddefault}{\updefault}{\color[rgb]{0,0,0}2}%
}}}}
\put(5446,-3616){\makebox(0,0)[lb]{\smash{{\SetFigFontNFSS{10}{12.0}{\familydefault}{\mddefault}{\updefault}{\color[rgb]{0,0,0}2}%
}}}}
\put(6886,-3616){\makebox(0,0)[lb]{\smash{{\SetFigFontNFSS{10}{12.0}{\familydefault}{\mddefault}{\updefault}{\color[rgb]{0,0,0}2}%
}}}}
\put(6166,-3616){\makebox(0,0)[lb]{\smash{{\SetFigFontNFSS{10}{12.0}{\familydefault}{\mddefault}{\updefault}{\color[rgb]{0,0,0}1}%
}}}}
\put(5761,-3121){\makebox(0,0)[lb]{\smash{{\SetFigFontNFSS{12}{14.4}{\familydefault}{\mddefault}{\updefault}{\color[rgb]{0,0,0}\normalsize{$|L_0(P^*)|=1$, $|L_1(P^*)|=2, |L_2(P^*) |=2$}}%
}}}}
\put(2881,-3121){\makebox(0,0)[lb]{\smash{{\SetFigFontNFSS{12}{14.4}{\familydefault}{\mddefault}{\updefault}{\color[rgb]{0,0,0}\normalsize{$|L_0(P)|=3$, $|L_1(P)|=1, |L_2(P) |=1$}}%
}}}}
\put(3331,-2176){\makebox(0,0)[lb]{\smash{{\SetFigFontNFSS{12}{14.4}{\familydefault}{\mddefault}{\updefault}{\color[rgb]{0,0,0}\normalsize{$P$}}%
}}}}
\put(6346,-2176){\makebox(0,0)[lb]{\smash{{\SetFigFontNFSS{12}{14.4}{\familydefault}{\mddefault}{\updefault}{\color[rgb]{0,0,0}\normalsize{$P^*$}}%
}}}}
\end{picture}%
}

It remains to see that the matching between opening and closing
half-chords that characterizes $\pi$ is forced by $P$. We will prove this recursively, by
matching opening chords  run by run, from left to right. That is, we
match the $m_0$ opening chords labelled $0$, then the
$m_1$ opening chords labelled 1, and so on. Assume we have
matched the first $m_0+m_1+\cdots +m_{i-1}$ opening chords.
For $0\le j\le k$, let $m_{i,j}$ be the number of elements of $P$ that have
level $i$ in $P$ and level $j$ in $P^*$. This is the number of chords
of $\pi$ with opening label $i$ and closing label $j$. Of course,
$m_{i,0}+\cdots + m_{i,k}=m_i$. 

Observe the following property:
\begin{quote}
($\star$) An involution $\pi$ avoids neighbour nestings if and only if, for
every opening run found at positions $i, i+1, \ldots, i+\ell$, one has
$\pi_i<\pi_{i+1} <\cdots < \pi_{i+\ell}$, and symmetrically, for every
closing run found at positions $i-\ell, \ldots, i-1, i$, one has 
$\pi_{i-\ell}<\cdots <\pi_{i-1} < \pi_{i}$.
\end{quote}

 This  property implies that the $m_{i,k}$
first (i.e., leftmost) opening chords labelled $i$ must be matched with
closing chords labelled $k$, the $m_{i,k-1}$
next opening chords labelled $i$ must be matched with closing chords
labelled $k-1$, and so on.  The second part of property~($\star$)
then forces the choice of the $m_{i,j}$  closing chords that will
be matched with opening chords labelled $i$: they are the leftmost
unmatched closing chords labelled $j$.  The matching of
half-arches is thus forced, and $\pi$ can be completely
reconstructed from $P$. Hence the restriction of $\Omega$ to
$\Invs{2n}$ is injective.

Let us illustrate the matching procedure by completing our running
example.
 For the above poset $P$, we find
$m_{0,2}= 2$, $m_{0,1}= 1$, $m_{0,0}= 0$, which allows us to match the
chords of the first opening run (equivalently, the opening chords
labelled $0$):

\vskip 2mm
  \centerline{\scalebox{1}{\includegraphics{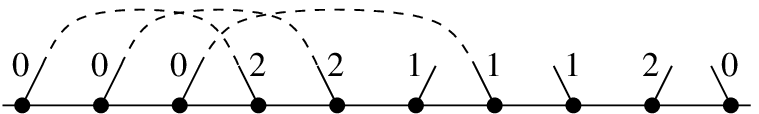}}}

Then, $m_{1,2}= 0$, $m_{1,1}= 1$, $m_{1,0}= 0$, which forces the
matching of the (unique) opening chord labelled 1:
\vskip 2mm

  \centerline{\scalebox{1}{\includegraphics{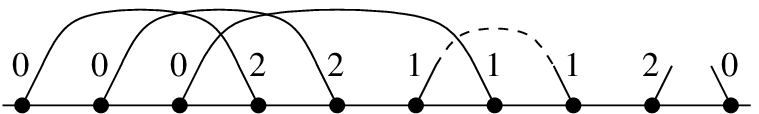}}}

Finally,
 $m_{2,2}= 0$, $m_{2,1}= 0$, $m_{2,0}= 1$, and we recover the
involution with no neighbour nesting shown in Figure~\ref{fig:chords}:
\vskip 2mm

  \centerline{\scalebox{1}{\includegraphics{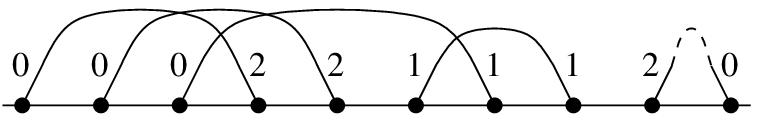}}}

\end{proof}

\noindent{\bf Remarks}\\
1. It follows from the proof of Theorem~\ref{thm:inv-poset} that, given any
collection of intervals with distinct endpoints whose relative order
is $P$, the transformations in Figure~\ref{fig:swapping}, applied in any order,
will yield ultimately the chord diagram of the involution
$\Omega^{-1}(P)$.  Note that these transformations boil down to
conjugating a fixed point free involution by the elementary
transposition $(i,i+1)$.

\noindent
2. We have worked out the recursive description of
involutions of $\Invs{2n}$ that corresponds, via the transformation
$\Omega$, to the recursive construction  of 
\tpt-free posets described in Section~\ref{sec:poset}, but it is
rather involved~\cite{CDK}. 

\noindent
3. The correspondence $\Omega$ allows one to read from
an involution $\pi\in \Invs{2n}$  the statistics  defined in
Section~\ref{sec:stat} for the poset $P=\Omega(\pi)$. For instance,
the number of minimal elements in $P$ is the length of the first
opening run of $\pi$.  Symmetrically, the number of maximal elements
of $P$ is the length of the last closing run of $\pi$. We have already
discussed how the distribution of levels of the elements of $P$ can be read from
$\pi$. Finally, there is a natural analogue on involutions for the
number of components of a poset.

%%%%%%%%%%%%%%%%%%%%%%%%%%%%%%%%%%%%%%%%%%%%%%%%%%%%%%%%%%%%%%%%%
\section{Final questions and remarks}\label{sec:open}
%%%%%%%%%%%%%%%%%%%%%%%%%%%%%%%%%%%%%%%%%%%%%%%%%%%%%%%%%%%%%%%%

\begin{question}
  Is there a simple graphical construction
on the dot diagram of a permutation
  in $\R_{n}$ that gives bijectively
an unlabeled {\tpt}-free poset on  $n$ elements?

A simple idea would be to view the dots of the diagram
 as a poset under the standard product order on $\NN^2$, as is done
 in~\cite{bmb}. 
For $n\leq 4$ 
 the posets associated with permutations of $ \R_n$
are exactly the  unlabeled {\tpt}-free posets of size $n$.
However, for $n=5$ the poset corresponding to the permutation $\pi =
41523\in\R_5$ contains an induced copy of \tptp.  This is illustrated in the
diagram below.\medskip
$$
\includegraphics[scale=0.6]{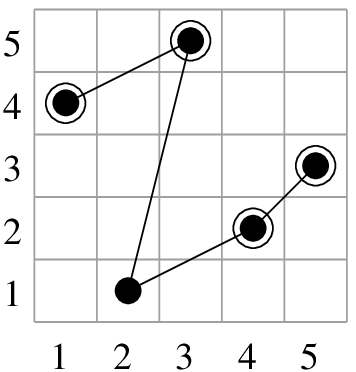} 
$$
\end{question}

\begin{question}
  Ascent sequences are special inversion tables. Turn these inversion
  tables into permutations in the two standard ways
  (see~\cite[p.~20-21]{stanley}). Is there a simple characterisation
  of those sets of permutations?
\end{question}

\begin{question}
A simple involution
acts on the set of \tpt-free posets: duality, 
or order-reversion. In terms of chord diagrams, this corresponds
to taking the mirror image of a diagram. What is the corresponding
transformation  
on permutations of $\R$? For instance, the permutation associated with
the poset $P$ of Example~\ref{ex:rem} is $31746825$, while the permutation
associated with the dual poset $P^*$ is $41726583$.  
\end{question}

\noindent {\bf Acknowledgment.}
Thanks to Henning \'Ulfarsson for pointing out that bivincular
permutations form a quasigroup, rather than a group as we incorrectly
claimed in an 
earlier
 draft.

%%%%%%%%%%%%%%%%%%%%%%%%%%%%%%%%%%%%%%%%%%%%%%%%%%%%%%%%%%
% 


\begin{thebibliography}{99}

\bibitem{albert} M. H. Albert, M. D. Atkinson, and R. Brignall,
 Permutation classes of polynomial growth,
{\em Ann. Comb.} {\bf  11} (2007) 249--264.


\bibitem {atkinson}
M. D. Atkinson and T. Stitt, Restricted permutations and the wreath product,
{\em Discrete Math.} {\bf 259} (2002) 19--36.

\bibitem{BABSON_EINAR} E. Babson and E. Steingr\'imsson, Generalized
  permutation patterns and a classification of the Mahonian
  statistics, {\em{S\'em. Lothar. Combin.}}  {\bf{44}} (2000)
  Art. B44b, 18 pp.

\bibitem{bogart} K. P. Bogart, An obvious proof of Fishburn's 
interval order theorem, {\em{Discrete Math.}} {\bf{118}} (1993) 239--242.

\bibitem{bmb} M. Bousquet-M\'elou and S. Butler, Forest-like
  permutations, {\em{Ann. Comb.}} {\bf{11}} (2007) 335--354.

\bibitem{chen} W. Y. C. Chen, E. Y.  Deng,
  R. R. Du, R. P. Stanley, C. H. Yan, Crossings and nestings of
  matchings and partitions, {\em Trans. Amer. Math. Soc.} {\bf  359}
  (2007) 1555--1575.  

\bibitem{CDK} A. Claesson, M. Dukes and S. Kitaev,
  A direct encoding of Stoimenow's matchings as ascent sequences,
  arXiv:0910.1619.

\bibitem{ZAHAR} M. H. El-Zahar, Enumeration of ordered sets, in: I.
  Rival (Ed.), {\em{Algorithms and Order}}, Kluwer Academic
  Publishers, Dordrecht, 1989, 327--352.


\bibitem{FISH_BOOK} P. C. Fishburn, {\em{Interval Graphs and Interval
    Orders}}, Wiley, New York, 1985.
 
\bibitem{FISH_OPER} P. C. Fishburn, Intransitive indifference in
  preference theory: a survey, {\em{Oper. Res.}} {\bf{18}} (1970)
  207--208.

\bibitem{fishburn} P. C. Fishburn, Intransitive indifference with
  unequal indifference intervals, {\em{J. Math. Psych.}} {\bf{7}}
  (1970) 144--149.

\bibitem{haxell}
P. E. Haxell, J. J. McDonald and S. K. Thomasson, Counting interval
orders,
 {\em{Order}} {\bf 4} (1987) 269--272.


\bibitem{smkhamis} S. M. Khamis, Height counting of unlabeled interval
  and $N$-free posets, {\em{Discrete Math.}} {\bf{275}} (2004)
  165--175.

\bibitem{krattenthaler} C. Krattenthaler, Growth diagrams, and
  increasing and decreasing chains in fillings of Ferrers shapes, 
{\em{Adv. Appl. Math.}} {\bf 37} (2006) 404--431.  


\bibitem{sloane} S. Plouffe and N. J. A. Sloane, The  Encyclopedia of
  Integer Sequences, Academic Press Inc.,
  San Diego, 1995.  Electronic version available at {\tt
    www.research.att.com/}$\!\sim${\tt njas/sequences/}.

\bibitem{lara} L.~Pudwell, Enumeration Schemes for Pattern-Avoiding
  Words and Permutations, PhD Thesis, Rutgers University (2008).

\bibitem{SKANDERA} M. Skandera, A characterization of $(3+1)$-free
  posets, {\em{J. Combin. Theory Ser. A}} {\bf{93}} (2001)
  231--241.

\bibitem{stanley} 
R.~P. Stanley,
\newblock {\em Enumerative combinatorics {V}ol. 1}, volume~49 of {\em
  Cambridge Studies in Advanced Mathematics},
\newblock Cambridge University Press, Cambridge, 1997.


\bibitem{stanley-vol2}
R.~P. Stanley,
\newblock {\em Enumerative combinatorics {V}ol. 2}, volume~62 of {\em
  Cambridge Studies in Advanced Mathematics},
\newblock Cambridge University Press, Cambridge, 1999.

\bibitem{stoim} A. Stoimenow, 
Enumeration of chord diagrams and an upper bound for
Vassiliev invariants, {\em{J. Knot Theory Ramifications}} {\bf{7}}
(1998) 93--114.

\bibitem{west-trees}
J.~West,
\newblock Generating trees and the {C}atalan and {S}chr\"oder numbers.
\newblock {\em Discrete Math.} {\bf{146}} (1995) 247--262.

\bibitem{wimp-zeil}
J. Wimp and D. Zeilberger,
    Resurrecting the asymptotics of linear recurrences,
   {\em{J. Math. Anal. Appl.}}  {\bf 111} 
     (1985) 162--176.

\bibitem{zagier} D. Zagier,  Vassiliev invariants and a strange
identity related to the Dedeking eta-function, {\em{Topology}} {\bf 40}
(2001) 945--960.

\bibitem{zeil-umbral}
 D. Zeilberger,
The Umbral Transfer-Matrix Method: I. {F}oundations,
{\em{J. Comb. Theory Ser. A}}  {\bf 91} (2000) 451--463.
\end{thebibliography}
\end{document}